\documentclass[10pt,reqno]{amsart}

\usepackage{amssymb}
\usepackage{amsfonts}
\usepackage{amsthm}
\usepackage{amsmath}

\numberwithin{equation}{section}
\allowdisplaybreaks

\newtheorem{theorem}{Theorem}[section]
\newtheorem{proposition}[theorem]{Proposition}

\newtheorem{lemma}[theorem]{Lemma}

\newtheorem{conjecture}[theorem]{Conjecture}
\theoremstyle{definition}
\newtheorem{definition}[theorem]{Definition}

\theoremstyle{remark}

\newcommand{\bthm}{\begin{theorem}}
\newcommand{\ethm}{\end{theorem}}
\newcommand{\bdfn}{\begin{definition} \em}
\newcommand{\edfn}{\end{definition}}
\newcommand{\bprop}{\begin{proposition}}
\newcommand{\eprop}{\end{proposition}}
\newcommand{\tnorm}{\textnormal}
\newcommand{\R}{\mathbb{R}}
\newcommand{\N}{\mathbb{N}}
\newcommand{\Z}{\mathbb{Z}}

\newcommand{\beq}{\begin{equation}}
\newcommand{\eeq}{\end{equation}}

\newcommand{\lsm}{\lesssim}

\newcommand{\eps}{\varepsilon}

\newcommand{\scriptJ}{\mathcal{J}}
\newcommand{\scriptK}{\mathcal{K}}

\newcommand{\qtq}[1]{\quad\text{#1}\quad}
\DeclareMathOperator*{\interior}{int}

\newcounter{smalllist}

\begin{document}

\title[The X-ray transform restricted to polynomial curves]{Uniform estimates for the X-ray transform restricted to polynomial curves}
\author{Spyridon Dendrinos}
\address{Department of Mathematics and Statistics, University of Jyv\"{a}skyl\"{a}, P.O. Box 35 (MaD), 40014 Jyv\"{a}skyl\"{a}, Finland.} \email{spyridon.dendrinos@jyu.fi}
\author{Betsy Stovall}
\address{Mathematics Department\\ University of California\\ Los Angeles, CA 90095-1555}
\email{betsy@math.ucla.edu}

\begin{abstract}
We establish near-optimal mixed-norm estimates for the X-ray transform restricted to polynomial curves with a weight that is a power of the affine arclength.  The bounds that we establish depend only on the spatial dimension and the degree of the polynomial.  Some of our results are new even in the well-curved case.  
\end{abstract}

\maketitle

%%%%%%%%%%%%%%%%%%%%%%%%%%%%%%%%%%%%%%%%%%%%%%%%%
%%%%%%%%%%%%%%%%%%%%%%%%%%%%%%%%%%%%%%%%%%%%%%%%%

\section{Introduction} \label{S:Introduction}

%%%%%%%%%%%%%%%%%%%%%%%%%%%%%%%%%%%%%%%%%%%%%%%%%
%%%%%%%%%%%%%%%%%%%%%%%%%%%%%%%%%%%%%%%%%%%%%%%%%

The X-ray transform, which we denote by $X_{\rm{full}}$, is a linear operator mapping functions on $\R^d$ to functions on the set $\mathcal{G}$ of all lines in $\R^d$ via
$$
X_{\rm{full}}f(l) = \int_l f, 
$$
where the integral is taken with respect to Lebesgue measure.  As $\mathcal G$ is of dimension $2(d-1)$, this operator is overdetermined whenever $d \geq 3$; this motivates the consideration of the restriction of $X_{\rm{full}}$ to the set of lines whose directions are parametrized by a fixed curve $\gamma:\R \to \R^{d-1}$.  The resulting restricted X-ray transform, after reparametrizing, maps functions on $\R^d$ to functions on $\R^d$ via
$$
X^\gamma f(t,y) = \int_\R f(s,y+s\gamma(t))\, ds.
$$
Because it is natural to bound $X_{\rm{full}}$ in mixed norm spaces (indeed, the conjectured mixed-norm bounds for $X_{\rm{full}}$ are known to imply the Kakeya conjecture--\cite{WolffBilin}), we seek mixed norm estimates for $X^\gamma$ of the form $\|X^\gamma f\|_{L^q(L^r)} \lesssim \|f\|_{L^p}$, where $L^q(L^r)$ is the space whose norm is given by
$$
\|g\|_{L^q(L^r)} = \bigl(\int_\R\bigl(\int_{\R^{d-1}} |g(t,y)|^r\, dy\bigr)^{\frac{q}r}\, dt\bigr)^{\frac1r}.
$$

It has been known for some time (see for instance \cite{CE2}, \cite{GS}) that the mapping properties of $X^\gamma$ depend on the torsion 
\begin{equation} \label{E:def LP}
L_\gamma = \det(\gamma',\gamma'',\ldots,\gamma^{(d-1)}),
\end{equation} 
the best estimates being possible in the well-curved case, where the torsion never vanishes.  Motivated by recent work on convolution and Fourier restriction operators, we seek to counteract potential degeneracies of curvature.  We accomplish this in the case of polynomial curves by multiplying $X^\gamma$ by a weight that is a power of the affine arclength, and obtain bounds that depend only on the dimension and the degree of the polynomial.  Even for the localized operator, this weight turns out to be optimal in a sense that will be made precise later. 

%%%%%%%%%%%%%%%%%%%%%%%%%%%%%%%%%%%%%%%%%%%%%%%
%%%%%%%%%%%%%%%%%%%%%%%%%%%%%%%%%%%%%%%%%%%%%%%

\section{Background and statement of results}  \label{S:Background}

%%%%%%%%%%%%%%%%%%%%%%%%%%%%%%%%%%%%%%%%%%%%%%%
%%%%%%%%%%%%%%%%%%%%%%%%%%%%%%%%%%%%%%%%%%%%%%%

For the purposes of this discussion, we denote by $X^\gamma_{\rm{loc}}$ the localized operator, given by
$$
X^\gamma_{\rm{loc}} f(t,y) = \int_{\R} f(s,y+s\gamma(t))a(s,t)\, ds,
$$
for some compactly supported $a$.  Because the torsion governs the mapping properties of $X^\gamma$ and $X^\gamma_{\rm{loc}}$, a useful model is $\gamma(t)=P_0(t) = (t,t^2,\ldots,t^{d-1})$, the so-called moment curve.  It is conjectured (necessity was proved by Erdo\u{g}an in \cite{Erdogan}) that the X-ray transform restricted to the moment curve satisfies
\begin{equation} \label{E:loc moment}
\|X^{P_0}_{\rm{loc}}f\|_{L^q(L^r)} \lesssim \|f\|_{L^p}
\end{equation}
if and only if $p$, $q$, and $r$ satisfy
\begin{gather}
\label{E:pqr1}  dp^{-1}\leq (d-1)r^{-1}+1,\\
\label{E:pqr2}  d(d-1)p^{-1} \leq 2q^{-1}+d(d-1)r^{-1},\\
\label{E:pqr3} (d-2)(d+1)p^{-1} \leq d(d-1)r^{-1}. 
\end{gather}
Without the localization in place, scaling dictates that $X^{P_0}:L^p \to L^q(L^r)$ if and only if \eqref{E:pqr1} and \eqref{E:pqr2} hold with equality and $X^{P_0}_{\rm{loc}}$ maps $L^p$ into $L^q(L^r)$.  

For a general curve $\gamma:\R \to \R^{d-1}$, the torsion \eqref{E:def LP} may vanish at some points, and a natural question is whether it is possible to compensate for such degeneracies of curvature.  This question was first formulated in the context of the adjoint Fourier restriction operators by Drury and Marshall, who in \cite{DM85, DM87} asked whether the operators
\begin{equation} \label{E:restriction}
\mathcal{E}^\gamma f(x) = \int_\R e^{ix\cdot \gamma(t)}f(t)\, |L_\gamma(t)|^{\frac2{d(d-1)}}\,dt, \quad x \in \R^{d-1}
\end{equation}
satisfy $L^p(\R) \to L^q(\R^{d-1})$ bounds with $p$ and $q$ independent of the curve $\gamma$.  This seems to be the case, at least for sufficiently nice curves, as has been seen in \cite{BOS2, BOS1, BOS3, DFW, DM11, DW} and other articles.  Later, Drury (\cite{Drury}) asked the same question about the convolution operator
\begin{equation} \label{E:convolution}
T^\gamma f(x) = \int_\R f(x-\gamma(t))\, |L_\gamma(t)|^{\frac2{d(d-1)}}\,dt, \quad x \in \R^{d-1};
\end{equation}
this was settled in the affirmative for polynomial curves in \cite{DLW, O1, StovallJFA}.  (Results on a different class of curves may be found in \cite{O2}.)

Thus in both the restriction and convolution cases, it has been seen that the natural choice to compensate for degeneracies of curvature is the affine arclength measure $\lambda$, which in parametrized form $\gamma^*\lambda$ is given by
\begin{equation} \label{E:aff arclength meas}
d\gamma^*\lambda(t) = |\det(\gamma'(t),\ldots,\gamma^{(d-1)}(t))|^{\frac2{d(d-1)}}\, dt.
\end{equation}
Moreover, affine arclength is extremely well-behaved under affine transformations (cf.\ Lemma~\ref{L:invariance}), and so it is reasonable to expect uniform bounds over certain classes of curves, such as polynomials of a fixed degree.  In the case of restricted X-ray transforms, this suggests the following.

%%%%%%%%
\begin{conjecture} \label{conjecture}
Let $d \geq 3$ and let $P:\R \to \R^{d-1}$ be a polynomial of degree $N$.  Then for any $p,q,r$ satisfying \eqref{E:pqr1} and \eqref{E:pqr2}, with equality in each, and \eqref{E:pqr3}, we have
\begin{equation} \label{E:conj}
\|X^P f\|_{L^q(L^r; d\gamma^*\lambda)} \leq C \|f\|_{L^p},
\end{equation}
for all $f \in L^p$.  The constant $C$ depends only on $d$, $N$, and $\theta$.  Furthermore, if $P$ is a fixed polynomial curve and $L_P$ is not identically zero, then these are the only exponents for which \eqref{E:conj} can hold.  
\end{conjecture}
%%%%%%%%

Here $\lambda$ is the measure in \eqref{E:aff arclength meas}, and we use the notation
$$
\|g\|_{L^q(L^r; d\gamma^*\lambda)} = \bigl(\int_\R \bigl(\int_{\R^{d-1}}|g(t,y)|^r\, dy\bigr)^{\frac qr} |L_\gamma(t)|^{\frac2{d(d-1)}}\, dt\bigr)^{\frac1q}.
$$

The necessity portion of this conjecture may be proved by modifying the proof of necessity given in \cite{Erdogan} for the well-curved case. 

Before stating our result, we observe that it is possible to rephrase Conjecture~\ref{conjecture} slightly.  The points $p,q,r$ under consideration are precisely those which may be written as
\begin{equation} \label{E:def p}
\bigl(p_{\theta}^{-1},q_{\theta}^{-1},r_{\theta}^{-1}\bigr) = \left(1-\theta + \frac{\theta d}{d+2},\,\,\, \frac{\theta d}{d+2},\,\,\, 1-\theta + \frac{\theta(d^2-d-2)}{d^2+d-2}\right),
\end{equation}
for some $0 \leq \theta \leq 1$.  Thus \eqref{E:conj} when $(p,q,r)=(p_\theta,q_\theta,r_\theta)$ is equivalent to the bound 
\begin{equation} \label{E:strong p theta bound}
\|X^P_{\theta}f \|_{L^{q_\theta}(L^{r_\theta})} \le C  \|f\|_{L^{p_\theta}},
\end{equation}
where
\begin{equation} \label{E:def X}
X^P_{\theta} f(t,y) = \int_{\R} f(s,y+sP(t))|L_P(t)|^{\frac{2\theta}{(d+2)(d-1)}}\, ds,
\end{equation}
and the conjecture is that \eqref{E:strong p theta bound} holds for all $0 \leq \theta \leq 1$.  In this article, we prove the following.

%%%%%%%
\begin{theorem} \label{T:main}  Let $d \geq 3$ and let $P:\R \to \R^{d-1}$ be a polynomial of degree $N$.  Then for $0 \leq \theta < 1$,  \eqref{E:strong p theta bound} holds for all $f$ in $L^{p_\theta}$.  At the endpoint $\theta = 1$, we have the restricted weak-type bound
\begin{equation} \label{E:weak p1 bound}
|\langle X^P_1 \chi_E, \chi_F \rangle| \le C |E|^{\frac1{p_1}}\|\chi_F\|_{L^{q_1'}(L^{r_1'})},
\end{equation}
for all measurable sets $E, F \subset \R^d$.  The constants $C$ in \eqref{E:strong p theta bound} and \eqref{E:weak p1 bound} depend only on $d$, $N$, and $\theta$.
\end{theorem}
%%%%%%%%%%

Thus the conjecture holds except possibly at the endpoint $(p_1,q_1,r_1)$.  Inequality \eqref{E:weak p1 bound} is a restricted weak type version of \eqref{E:strong p theta bound}. The authors believe that their analysis could be modified to obtain a restricted weak type version of \eqref{E:conj} at the endpoint, but \eqref{E:weak p1 bound} seems to have a slightly simpler proof.  

In addition to Theorem~\ref{T:main} being nearly optimal in terms of the exponents involved, we show in Proposition~\ref{P:optimality} that $\lambda$ is in some sense the largest measure for which even the restricted weak type estimates in Conjecture~\ref{conjecture} can hold.  

We will turn in a moment to a discussion of some prior work concerning the $X^P$, but first, a word on the $(p,q,r)$ under consideration.  Three values of $\theta$ carry particular significance in our analysis, and we record the corresponding triples here.  Naturally, two of these values are the endpoints $\theta = 0,1$.  We have
$$
\bigl(p_0,q_0,r_0\bigr) = \bigl(1,\infty,1\bigr), \qquad \bigl(p_1,q_1,r_1\bigr) = \left(\frac{d+2}d, \frac{d+2}d, \frac{d^2+d-2}{d^2-d-2}\right).
$$
The third value, which we denote by $\theta_0$, is the unique parameter satisfying $q_{\theta_0} = r_{\theta_0}$, or equivalently $L^{q_{\theta_0}}(L^{r_{\theta_0}}) = L^{q_{\theta_0}}$.  It is easy to check that $\theta_0 = \frac{(d+2)(d-1)}{d^2+d}$ and
$$
\bigl(p_{\theta_0},q_{\theta_0},r_{\theta_0}\bigr) = \left(\frac{d(d+1)}{d^2-d+2},\frac{d+1}{d-1},\frac{d+1}{d-1}\right).
$$

In the case of the moment curve, weaker versions of \eqref{E:strong p theta bound} are known in all dimensions.  These are due to Wolff in \cite{WolffBilin} when $d=3$, to Erdo\u{g}an in \cite{Erdogan} when $d=4,5$, and to Christ--Erdo\u{g}an in \cite{CE1} when $d\geq6$.  Earlier work concerning non-mixed estimates was carried out in \cite{GS} and \cite{OberlinXray}.  In \cite{CE1, Erdogan, WolffBilin}, \eqref{E:loc moment} was also proved for all $p,q,r$ satisfying \eqref{E:pqr1}, \eqref{E:pqr2}, and \eqref{E:pqr3} with strict inequality in each.  The strong type bound (again in the case of the moment curve) was proved by Laghi when $\theta=\theta_0$ for $d \geq 3$.  By interpolation with the trivial $L^1 \to L^\infty(L^1)$ estimate, Conjecture~\ref{conjecture} has thus been verified in the case of the moment curve when $d \geq 3$ and $0 \leq \theta \leq \theta_0$.  Thus even in the well-curved case, some of our results are new.

For more general curves, the endpoint restricted-weak type (unweighted, hence depending on $P$), $L^{p,1} \to L^{q,\infty}$ estimates for $X^P_{\rm{loc}}$ follow from the work of Gressman in \cite{Gressman}.  It seems likely that all of the (again unweighted) restricted weak type $L^p \to L^q(L^r)$ estimates for $X^P_{\rm{loc}}$ may be proved by combining the techniques in \cite{Gressman} with those in \cite{CE1, CE2}, but the authors have not undertaken to verify this.  For $(p^{-1},q^{-1},r^{-1})$ lying in the interior of the conjectured region of $L^p \to L^q(L^r)$ boundedness, the result was established in \cite{CE2}.  Our theorem differs from all of these results in two significant ways.  First, the results of \cite{Gressman} do not involve a weight, and so the exponents involved and the bounds obtained depend on the particular curve under consideration.  Second, we establish strong type estimates in many cases where solely applying the results of \cite{CE1} and \cite{Gressman} would yield restricted weak type bounds.  We will say more about these issues in a moment. 

We remark that there is an equivalent point of view, namely the double-fibration formulation, which originated in \cite{GGS, GuillemanSternberg} and which was discussed at length in \cite{TW}. More specifically, by duality, Theorem~\ref{T:main} implies that for any measurable set $\Omega \subset \R^{d+1}$, we have
\begin{equation} \label{pi12}
\int \chi_{\Omega}(s,t,x)\, |L_P(t)|^{\frac{2\theta}{(d+2)(d-1)}}\, ds\,dt\,dx \lesssim |\pi_1(\Omega)|^{\frac1{p_\theta}}\|\chi_{\pi_2(\Omega)}\|_{L^{q_\theta'}(L^{r_\theta'})},
\end{equation}
where the mappings $\pi_1$, $\pi_2:\R^{d+1} \to \R^d$ are defined by
\[
\pi_1(s,t,x)=(s,x+sP(t)), \quad \pi_2(s,t,x) = (t,x).
\]
Inequality (\ref{pi12}) can be regarded as an isoperimetric inequality for sets in $\R^{d+1}$ and thus may be of independent interest. We will, however, not elaborate further on this point of view.

\subsection*{Outline of proof.}  In Section~\ref{S:Preliminaries}, we set out some preliminaries and prove the invariance and optimality assertions made in the remarks above.  Our proof uses the method of refinements (cf.\ \cite{CCC}), and as such, we need lower bounds for the Jacobian determinants of certain maps that arise when we iterate; these are obtained in Section~\ref{S:Jacobians}.  In Section~\ref{S:RWT}, we prove the restricted weak type version of \eqref{E:strong p theta bound} for $\theta_0 \leq \theta \leq 1$.  To do this, we use the lower bounds from Section~\ref{S:Jacobians} as well as ideas adapted from \cite{CE1}.  We note that the presence of the affine arclength term means that even in the non-mixed case, these restricted weak type estimates do not follow directly from the results of \cite{Gressman}, which is why we use more explicit computations.  

It is not, to the authors' knowledge, known whether there is an analogue of Marcinkiewicz interpolation that could be used to prove the main theorem from this restricted weak type result, and so our work is not done.  In Section~\ref{S:Interpolation}, we prove a simple interpolation lemma, thereby obtaining improved, but non-optimal, bounds in the range $\theta_0 < \theta < 1$.  We also give a partial characterization of the quasi-extremizers for these bounds.  Finally, in Section~\ref{S:Strong type}, we complete the proof by adapting an argument of Christ in \cite{ChQex}, which has previously only been used in the non-mixed setting.  This adaptation uses the characterization of quasi-extremizers from Section~\ref{S:Interpolation} and seems to be the first time that such a result has been used in conjunction with the methods of \cite{ChQex} to prove strong type bounds (even in the non-mixed case).

\subsection*{Notation.} If $A$ and $B$ are two positive numbers, then we write $A \lesssim B$ to mean that $A \leq CB$, where the constant $C>0$ may change from line to line and depends only on $d$, $\theta$, and the degree $N$ of $P$.  By $A \sim B$, we mean $A \lesssim B$ and $B \lesssim A$.  We will occasionally write `$A \ll B$' as a hypothesis; this is just a short-hand for `$A \leq c B$ for some sufficiently small constant $c>0$ depending only on $d$, $N$, and $\theta$.'  Finally, we define $\Pi:\R^d\to\R$ to be the projection $\Pi(t,y) := t$.

\subsection*{Acknowledgments}  This collaboration was facilitated in part by the Research Support Fund of the Edinburgh Mathematical Society.  The second author is an NSF Postdoctoral Research Fellow.  We are grateful to Terence Tao for enlightening conversations regarding interpolation with mixed norms.

%%%%%%%%%%%%%%%%%%%%%%%%%%%%%%%%%%%%%%%%%%%%%%%%%
%%%%%%%%%%%%%%%%%%%%%%%%%%%%%%%%%%%%%%%%%%%%%%%%%

\section{Preliminary considerations} \label{S:Preliminaries}

%%%%%%%%%%%%%%%%%%%%%%%%%%%%%%%%%%%%%%%%%%%%%%%%%
%%%%%%%%%%%%%%%%%%%%%%%%%%%%%%%%%%%%%%%%%%%%%%%%%

We begin by noting that when $\theta = 0$, the strong-type bound in Theorem~\ref{T:main} is trivial by Fubini's theorem.  We record this observation here.

\begin{lemma} \label{L:L1 to LinftyL1}
The operator $X^P$ is a bounded operator from $L^1$ to $L^{\infty}(L^1)$ and satisfies the bound
$$
\|X^P f\|_{L^{\infty}(L^1)} \leq \|f\|_{L^1}, \qquad f \in L^1.
$$
\end{lemma}

For $\theta > 0$, we will not be able to compute the operator norm exactly, but as noted earlier, the operator norms of the $X_\theta^P$ are invariant under affine transformations and reparametrizations of $P$.  More concretely, we have the following

\begin{lemma} \label{L:invariance}
Let $A:\R^{d-1} \to \R^{d-1}$ be an invertible affine transformation, $A=B+c$ with $c\in\R^{d-1}$ and $B\in GL(d-1,\R)$, and let $\phi:\R \to \R$ be a diffeomorphism. Then, if $P:\R \to \R^{d-1}$ is a polynomial and $f \in L^{p_{\theta}}$ is not identically zero, we have
\begin{equation} \label{E:invariance}
\frac{\|X^P_{\theta} f\|_{L^{q_{\theta}}(L^{r_{\theta}})}}{\|f\|_{L^{p_{\theta}}}} = \frac{\|X^{\tilde P}_{\theta} \tilde f\|_{L^{q_{\theta}}(L^{r_{\theta}})}}{\|\tilde f\|_{L^{p_{\theta}}}},
\end{equation}
where $\tilde P = AP\circ \phi$ and $\tilde f(s,x) = f(s,A^{-1}y)$.  In particular, $X^P_\theta$ is a bounded operator from $L^{p_\theta}$ to $L^{q_\theta}(L^{r_\theta})$ if and only if $X^{\tilde P}_\theta$ is, and moreover, the two have the same operator norms.  
\end{lemma}

The proof is a routine computation, which we leave to the reader.

We now turn to the main goal of this section, which is to show that the weight that we use is optimal in the following sense.  

\begin{proposition} \label{P:optimality}  Let $0 \leq \theta \leq 1$.  Assume that $\rho$ is a positive Borel measure on $\R$ such that for any Borel sets $E$, $F$ in $\R^d$, we have
\begin{equation} \label{E:rho assumption}
\int_{\R^{d+1}} X^P \chi_E(t,y) \chi_F(t,y)\, d\rho(t)\, dy \leq C |E|^{\frac1{p_\theta}} \|\chi_F\|_{L^{q_\theta'}(L^{r_\theta'})}
\end{equation}
for some constant $C$.  Then $\rho$ is absolutely continuous with respect to Lebesgue measure and its Radon--Nikodym derivative satisfies
$$
\frac{d\rho}{dt}(t) \leq C_d C|L_P(t)|^{\frac{2\theta}{(d+2)(d-1)}},
$$
for some constant $C_d$ depending only on $d$. Here the constant $C$ is the same in both of the above inequalites.
\end{proposition}

We note that in the case of the convolution and Fourier restriction operators \eqref{E:restriction} and \eqref{E:convolution}, the analogous results are due to Oberlin in \cite{MichMath}.

We begin with the endpoint $\theta=0$.

\begin{proof}[Proof of Proposition~\ref{P:optimality} when $\theta=0$]  Let $t_0 \in \R$ and let $0 < \delta < 1$.  Define sets
$$
F = \{(t,y):|t-t_0|\leq \delta, \, |y| \leq 1\}, \quad E = \{(s,y+sP(t)):|s| \leq 1, \, |t-t_0| \leq \delta, \, |y| \leq 1\}.
$$
We then have that $\|\chi_F\|_{L^1(L^{\infty})} = 2\delta$ and for $\delta$ sufficiently small (depending on $t_0$), $|E| \leq C_d$.  Furthermore, if $(t,y) \in F$, then it is obvious that $X^P\chi_E(t,y) = 2$.  Therefore
$$
\rho([t_0-\delta,t_0+\delta]) \leq C_d \int X^P \chi_E(t,y) \chi_F(t,y)\, d\rho(t)\, dy,
$$
and so our assumption \eqref{E:rho assumption} implies that $\rho([t_0-\delta,t_0+\delta]) \leq C_d C \delta$.  This completes the proof.
\end{proof}

We now turn to the case when $\theta > 0$.  It is in this case that curvature plays a role, as we see in the following.

\begin{lemma} \label{L:flat case}  Suppose that $L_P \equiv 0$.  Then the image of $P$ lies in a hyperplane.  Moreover, \eqref{E:rho assumption} is only possible if $\rho \equiv 0$ or $\theta = 0$.  
\end{lemma}

The authors do not claim that this is a new result, but as we could not find a proof in the literature, we decided to include its simple proof for the convenience of the reader.

\begin{proof}[Proof of Lemma~\ref{L:flat case}]
We begin by proving the first conclusion.  For each $j=1,\ldots,d-1$, let 
$$
A_j = \{t \in \R : \dim(\tnorm{span}\{P'(t),\ldots,P^{(j)}(t)\}) = j\}.
$$
Then $A_1 \supset A_2 \supset \cdots \supset A_{d-1}$, and our hypothesis is that $A_{d-1} = \emptyset$.  If $A_1 = \emptyset$, then $P'(t) \equiv 0$, and the result is trivial.  Otherwise, we may fix $j$ ($1 \leq j \leq d-2$) to be the (unique) index such that $A_j \neq \emptyset$ and $A_{j+1} = \emptyset$.  Since $A_j$ is obviously open, it contains an open interval $J$.  

On $J$, we have that 
$$
P'(t) \wedge \cdots \wedge P^{(j)}(t) \wedge P^{(j+1)}(t) \equiv 0.
$$
Let us assume in addition that 
\begin{equation} \label{E:wedge to k}
P'(t) \wedge \cdots \wedge P^{(j)}(t) \wedge P^{(k)}(t) \equiv 0,
\end{equation}
for some $k \geq j+1$.  Differentiating \eqref{E:wedge to k}, we see that
$$
(P'(t) \wedge \cdots \wedge P^{(j+1)}(t) \wedge P^{(k)}(t)) + (P'(t) \wedge \cdots \wedge P^{(j)}(t) \wedge P^{(k+1)}(t)) \equiv 0
$$
on $J$.  Our hypotheses imply that $P^{(j+1)}(t)$ and $P^{(k)}(t)$ both lie in the span of $P'(t),\ldots,P^{(j)}(t)$ (which we have assumed are linearly independent) for every $t \in J$, so the first term in the above sum is identically zero.  This completes the inductive step, verifying that
$$
P'(t) \wedge \cdots \wedge P^{(j)}(t) \wedge P^{(k)}(t) \equiv 0
$$
on $J$ (and hence on $\R$) for each $k \in \N$.  

Without loss of generality, $0 \in J$ and $P(0) =0$.  For any $t \in \R$, we have
$$
P(t) = \sum_{n=1}^N \frac{t^n}{n!} P^{(n)}(0),
$$
and thus by the previous observation, $P$ lies in the subspace spanned by $P'(0),\ldots,\linebreak P^{(j)}(0)$.  Recalling that $j < d-1$, we have proved that the image of $P$ lies in a hyperplane.  

Applying a rotation if necessary, we may assume that $P \subset \R^{d-2} \times \{0\}$.  Given a bounded interval $I \subset \R$ and $\delta > 0$, we define sets
\begin{align*}
F&= \{(t,y',y_{d-1})\in \R\times\R^{d-2}\times\R : t \in I, \, |y'| \leq 1, \, |y_{d-1}| \leq \delta\} \\
E&= \{(s,x',x_{d-1})\in \R\times\R^{d-2}\times\R : |s| \leq 1, \, |x'| \leq 1+\sup_I|P(t)|, \, |x_{d-1}| \leq \delta\}.
\end{align*}
We observe that $|E| \leq C_{I,P,d} \delta$, $\int_F d\rho(t)\, dy \sim \rho(I)\delta$, and $\|\chi_F\|_{L^{q'}(L^{r'})} \sim |I|^{\frac1{q'}}\delta^{\frac1{r'}}$.  Additionally, for $(t,y) \in F$, we have that $X^P\chi_E(t,y) = 2$, so
$$
\rho(I)\delta \lesssim \int_{I\times \R^d} X^P \chi_E(t,y) \chi_F(t,y)\, d\rho(t)\, dy \lsm \delta^{\frac{1}{p}} |I|^{\frac{1}{q'}} \delta^{\frac{1}{r'}},
$$
and if $\theta > 0$ (so $r' < p' < \infty$), we see that $\rho(I) = 0$ by letting $\delta \searrow 0$.  This completes the proof of the lemma.
\end{proof}

We are finally ready to complete the proof of Proposition~\ref{P:optimality}.

%%%%%
\begin{proof}[Proof of Proposition~\ref{P:optimality} when $\theta>0$]  We begin by considering a point $t_0$ where $L_P(t_0) \neq 0$. Given $\delta> 0$, we define sets $E$ and $F$ by
\begin{align*}
E &:= \{(s,x+sP(t_0)) : |s| < 1, \,\, x = \sum_{j=1}^{d-1}v_jP^{(j)}(t_0), \text{\,\,where\,\,} |v_j| < 2\delta^j\}\\
F &:= \{(t,y) : |t-t_0| < \delta,  \,\, y = \sum_{j=1}^{d-1}v_j P^{(j)}(t_0), \text{\,\,where\,\,} |v_j| < \delta^j\}.
\end{align*}
It is easy to see that 
$$
|E| = 2^d \delta^{\frac{d(d-1)}2}|L_P(t_0)|, \qquad \int_F d\rho(t)\, dy = \rho([t_0-\delta,t_0+\delta]) \delta^{\frac{d(d-1)}2}|L_P(t_0)|,
$$
and moreover that 
$$
\|\chi_F\|_{L^{q'}(L^{r'})} = (\delta^{\frac{d(d-1)}2}|L_P(t_0)|)^{\frac1{r'}} \delta^{\frac1{q'}}.
$$
Since $P$ is a polynomial of degree $N$, we have
\begin{equation} \label{E:taylorP}
P(t) = P(t_0) + \sum_{j=1}^N \frac{(t-t_0)^j}{j!} P^{(j)}(t_0).
\end{equation}
By Cramer's rule, we have for $d \leq j \leq N$ that
$$
P^{(j)}(t_0) = \sum_{i=1}^{d-1} \frac{\det(P',\ldots,P^{(i-1)},P^{(j)},P^{(i+1)},\ldots,P^{(d-1)})(t_0)}{\det(P',\ldots,P^{(d-1)})(t_0)}P^{(i)}(t_0).
$$
Hence by \eqref{E:taylorP}, if $\delta$ is sufficiently small and  $|t-t_0| < \delta$, we have that
\[
P(t) = P(t_0) + \sum_{j=1}^d v_j P^{(j)}(t_0),
\]
with $|v_j| < 2\delta^j$.  Therefore 
\beq \label{implies}
(t,y) \in F \qtq{and} |s| \leq 1 \quad \Rightarrow \quad (s,y+sP(t)) \in E.
\eeq
This in turn implies that 
\[
\rho([t_0-\delta,t_0+\delta]) \leq \int X^P\chi_E(t,y)\chi_F(t,y)\, d\rho(t)\, dy \leq C_d C |E|^{\frac{1}{p_\theta}} \|\chi_F\|_{L^{q_\theta'}(L^{r_\theta'})}.
\]
After some algebra, we obtain
\begin{eqnarray*}
\rho([t_0-\delta,t_0+\delta]) & \leq &C_d C |L_P(t_0)|^{\frac{1}{p_\theta}-\frac{1}{r_\theta}} \delta^{\frac{d(d-1)}{2p_\theta} - \frac{d(d-1)}{2r_\theta} + 1 - \frac{1}{q_\theta}} \\
& = & C_d C|L_P(t_0)|^{\frac{2\theta}{(d+2)(d-1)}} \delta.
\end{eqnarray*}

The proposition then follows from the observation that for \eqref{E:rho assumption} to hold, $\rho(\{t_0\})=0$ for every $t_0 \in \R$ (in particular for those points satisfying $L_P(t_0)=0$).  This may be proved similarly to the proof of the proposition when $\theta=0$, and we leave the details to the reader.
\end{proof}

%%%%%%%%%%%%%%%%%%%%%%%%%%%%%%%%%%%%%%%%%%%%%%%%%
%%%%%%%%%%%%%%%%%%%%%%%%%%%%%%%%%%%%%%%%%%%%%%%%%

\section{Jacobian estimates} \label{S:Jacobians}

%%%%%%%%%%%%%%%%%%%%%%%%%%%%%%%%%%%%%%%%%%%%%%%%%
%%%%%%%%%%%%%%%%%%%%%%%%%%%%%%%%%%%%%%%%%%%%%%%%%

One of the main steps in our proof of Theorem~\ref{T:main} will be to prove that the operators $X_{\theta}$ satisfy the restricted weak-type bounds corresponding to \eqref{E:strong p theta bound}.  We will establish these bounds by using Christ's method of refinements (cf.\ \cite{CCC, CE1, TW}), which involves proving lower bounds for the volumes of certain sets obtained by iterating.  In order to do this, we will need to prove lower bounds for the Jacobian determinants of the maps that arise when we iterate. 

Before we begin, we record the formula
\[
X_\theta^* (s,x) = \int_\R f(t,x-sP(t))|L_P(t)|^{\frac{2\theta}{(d+2)(d-1)}} dt.
\] 
Here we have omitted the superscript $P$ from the operator, as we will continue to do for the remainder of the article.

Given base points $(s_0,x_0), (t_0,y_0) \in \R^d$, we define maps $\Phi^k_{(s_0,x_0)},\Psi^k_{(t_0,y_0)}:\R^k \to \R^d$ ($k=1,2,\dots,d$) by
\begin{align}
\label{E:def Phi even}
&\Phi^{2K}_{(s_0,x_0)}(t_1,s_1,\dots,t_K,s_K)   = \bigl(s_K,x_0 - \sum_{j=1}^K(s_{j-1}-s_j)P(t_j)\bigr), \\
\label{E:def Phi odd} 
&\Phi^{2K+1}_{(s_0,x_0)}(t_1,s_1,\dots,t_{K+1})  = \bigl(t_{K+1},x_0 - \sum_{j=1}^K(s_{j-1}-s_j)P(t_j) -s_KP(t_{K+1})\bigr), \\
\label{E:def Psi even}
&\Psi^{2K}_{(t_0,y_0)}(s_1,t_1,\dots,s_K,t_K)  = \bigl(t_K,y_0 + \sum_{j=1}^Ks_j(P(t_{j-1}) - P(t_j))\bigr), \\
\label{E:def Psi odd}
&\Psi^{2K+1}_{(t_0,y_0)}(s_1,t_1,\dots,s_{K+1}) = \bigl(s_{K+1},y_0 + s_1P(t_0)- \sum_{j=1}^K (s_j-s_{j+1})P(t_j)\bigr).
\end{align}

The main goal of this section will be to establish the following proposition, which relates the Jacobian determinants of $\Phi^d_{(s_0,x_0)}, \Psi^d_{(t_0,y_0)}$ to the torsion, $L_P=\det(P',\ldots,P^{(d-1)})$.  

\begin{proposition} \label{P:jacobian bounds}  Let $d = 2D$ be an even integer and $P: \R \to \R^{d-1}$ a polynomial of degree $N$.  Then there exists a decomposition $\R = \bigcup_{j=1}^{C_{N,d}} I_j$ into disjoint intervals such that for each $j$, the following hold:\\
\emph{(i)}  If $(s_0,s_1,\dots,s_D) \in \R^{D+1}$ and $(t_0,t_1,\dots,t_D) \in I_j^{D+1}$,  then
\begin{align} \label{E:Psi even}
&|\det\bigl(D\Psi^d_{(t_0,y_0)}(s_1,t_1,\dots,s_D,t_D)\bigr)| \\ \notag
&\qquad  \gtrsim \prod_{i=1}^{D-1}\bigl\{|s_{i+1}-s_i||L_P(t_i)|^{\frac2d}\prod_{\stackrel{0 \leq j \leq D}{j\neq i}}|t_j-t_i|^2\bigr\} |L_P(t_0)|^{\frac1d}|L_P(t_D)|^{\frac1d}|t_D-t_0| \\
 \label{E:Phi even}
& \begin{aligned}
 &|\det\bigl(D\Phi^d_{(s_0,x_0)}(t_1,s_1,\dots,t_D,s_D)\bigr)| \\
 &\qquad \gtrsim \prod_{i=1}^D \bigl\{|s_i-s_{i-1}||L_P(t_i)|^{\frac2d}\prod_{\stackrel{1 \leq j \leq D}{j\neq i}}|t_j-t_i|^2\bigr\} 
 \end{aligned}
\end{align}
\emph{(ii)} There exist a constant $A_j \geq 0$, an integer $K_j \in [0,C_{d,N}]$, and a real number $b_j \notin \interior I_j$ such that
\begin{equation} \label{E:poly is mono}
|L_P(t)| \sim A_j |t-b_j|^{K_j}, \qtq{for every $t \in I_j$.}
\end{equation}

If $d=2D+1$ is odd, then analogous statements hold, only we must modify the bounds in \eqref{E:Psi even}, \eqref{E:Phi even} to
\begin{align} \label{E:Psi odd}
&\begin{aligned}
&|\det\bigl(D\Psi^d_{(t_0,y_0)}(s_1,t_1,\dots,s_{D+1})\bigr)| \\ 
&\qquad  \gtrsim \prod_{i=1}^D\bigl\{|s_{i+1}-s_i||L_P(t_i)|^{\frac2d}\prod_{\stackrel{0 \leq j \leq D}{j \neq i}}|t_j-t_i|^2\bigr\} |L_P(t_0)|^{\frac1d}
\end{aligned}\\
\label{E:Phi odd}
&\begin{aligned}
& |\det\bigl(D\Phi^d_{(s_0,x_0)}(t_1,s_1,\dots,t_{D+1})\bigr)|  \\ 
& \qquad \gtrsim \prod_{i=1}^D \bigl\{|s_i-s_{i-1}| |L_P(t_i)|^{\frac2d}\prod_{\stackrel{1 \leq j \leq D+1}{j \neq i}}|t_j-t_i|^2 \bigr\} |L_P(t_{D+1})|^{\frac1d}.
\end{aligned}
\end{align}
Here again, $t_i \in I_j$, while $s_i \in \R$.
\end{proposition}

For the proof, we will find alternative expressions for the Jacobian determinants and then apply the following theorem from \cite{DW}.

\begin{theorem}[\cite{DW}] \label{T:DW}  If $Q:\R \to \R^d$ is a polynomial of degree $N$, then there exists a decomposition $\R = \bigcup_{j=1}^{C_{N,d}} I_j$ into disjoint intervals such that for each $j$, the following hold\\
(i)  If $(t_1,\dots,t_d) \in I_j^d$, then 
$$
|\det\bigl(Q'(t_1),\dots,Q'(t_d)\bigr)| \gtrsim \prod_{j=1}^d |L_Q(t_j)|^{\frac1d}\prod_{j < k}|t_j-t_k|.
$$
(ii)  There exist a constant $A_j \geq 0$, an integer $K_j \in [0,C_{d,N}]$, and a real number $b_j \notin \interior I_j$ such that 
$$
|L_Q(t)| \sim A_j |t-b_j|^{K_j}, \quad \text{for every $t \in I_j$.}
$$
\end{theorem}

We record some useful formulae here.

\begin{lemma} \label{L:jacobian formulae}
Let $Q(t) = (t,\int P(t)\, dt)$ be an antiderivative of $(1,P)$.  If $d=2D$ is even, then 
\begin{align} 
\label{E:DPsi even}
&\begin{aligned}
& \det\bigl(D\Psi^d_{(t_0,y_0)}(t_1,s_1,\dots,t_D,s_D)\bigr) = \\
&\qquad \pm \bigl\{\prod_{i=1}^{D-1}(s_{i+1}-s_i)\bigr\}\bigl\{\prod_{j=D+1}^{2D-1}\partial_j|_{t_j = t_{j-D}}\bigr\} \det\bigl(Q'(t_0),\dots,Q'(t_{2D-1})\bigr)
\end{aligned}\\
\label{E:DPhi even}
&\begin{aligned}
& \det\bigl(D\Phi^d_{(s_0,x_0)}(t_1,s_1,\dots,t_D,s_D)\bigr) = \\
 &\qquad \pm \bigl\{\prod_{i=1}^D(s_i-s_{i-1})\bigr\}\bigl\{\prod_{j=D}^{2D} \partial_j|_{t_j = t_{j-D}}\bigr\} \det\bigl(Q'(t_1),\dots,Q'(t_{2D})\bigr).
 \end{aligned}
\end{align}
If $d=2D+1$ is odd, then
\begin{align} 
\label{E:DPsi odd}
&\begin{aligned}
&\det\bigl(D\Psi^d_{(t_0,y_0)}(s_1,t_1,\dots,s_D,t_D,s_{D+1})\bigr) = \\
 &\qquad \pm \bigl\{\prod_{i=1}^D(s_{i+1}-s_i)\bigr\} \bigl\{\prod_{j=D+1}^{2D}\partial_j|_{t_j = t_{j-D}}\bigr\}  \det\bigl(Q'(t_0),\dots,Q'(t_{2D})\bigr) 
 \end{aligned}\\
\label{E:DPhi odd}
&\begin{aligned}
& \det\bigl(D\Phi^d_{(s_0,x_0)}(t_1,s_1,\dots,t_D,s_D,t_{D+1})\bigr) = \\
&\qquad \pm \bigl\{\prod_{i=1}^D(s_i-s_{i-1})\bigr\} \bigl\{\prod_{j=D+2}^{2D+1} \partial_j|_{t_j = t_{j-D-1)}}\bigr\} \det(Q'(t_1),\dots,Q'(t_{2D+1})).
\end{aligned}
\end{align}
\end{lemma}

\begin{proof}[Proof of Lemma~\ref{L:jacobian formulae}] 
We will give the proof of \eqref{E:DPsi even} only, the remaining formulas having similar derivations.  We compute:
\begin{align*}
& \det\bigl(D\Psi^d_{(t_0,y_0)}(t_1,s_1,\dots,t_D,s_D)\bigr) = \left|\begin{array}{cc}  0 & (s_1-s_2)P'(t_1) \\ \vdots & \vdots \\ 0 & (s_{D-1}-s_D)P'(t_{D-1}) \\ 1 & s_DP'(t_D) \\ 0 & P(t_1) - P(t_0) \\ \vdots & \vdots \\ 0 & P(t_D) - P(t_{D-1})  \end{array}\right| \\
& \qquad    = \pm \bigl\{\prod_{i=1}^{D-1}(s_{i+1}-s_i)\bigr\}\left|\begin{array}{c} P(t_1) - P(t_0) \\ \vdots \\  P(t_D) - P(t_{D-1}) \\ P'(t_1) \\ \vdots \\ P'(t_{D-1}) \end{array}\right| \\
&\qquad = \pm \bigl\{\prod_{i=1}^{D-1}(s_{i+1}-s_i)\bigr\} \left|\begin{array}{cc} 1 & P(t_0) \\ \vdots & \vdots \\ 1 & P(t_D) \\ 0 & P'(t_1) \\ \vdots & \vdots \\ 0 & P'(t_{D-1}) \end{array}\right| \\
&\qquad = \pm \bigl\{\prod_{i=1}^{D-1}(s_{i+1}-s_i)\bigr\} \bigl\{\prod_{j=D+1}^{2D-1} \partial_j|_{t_j = t_{j-D}}\bigr\} \left|\begin{array}{cc} 1 & P(t_0) \\ \vdots & \vdots \\ 1 & P(t_{2D-1}) \end{array}\right|. 
\end{align*}
\end{proof}

Now we are ready to begin the proof of Proposition~\ref{P:jacobian bounds}.  

\begin{proof}[Proof of Proposition~\ref{P:jacobian bounds}]  We will give only the proof of \eqref{E:Psi even}.  It suffices to establish a lower bound for 
\begin{equation} \label{E:DQ}
\bigl\{\prod_{j=D+1}^{2D-1} \partial_j|_{t_j = t_{j-D}}\bigr\} \det\bigl(Q'(t_0),\dots,Q'(t_{2D-1})\bigr).
 \end{equation}
Since the determinant on the right of \eqref{E:DQ} is an anti-symmetric polynomial, it may be factorized as 
\begin{equation} \label{E:dif1}
\det\bigl(Q'(t_0),\dots,Q'(t_{2D-1})\bigr) = 
\bigl\{\prod_{0\le i<j\le2D-1}(t_j-t_i)\bigr\} J(t_0,\ldots ,t_{2D-1}), 
\end{equation}
for some other polynomial $J$ of $2D$ variables.  Additionally, by Theorem~\ref{T:DW} and the trivial identity $L_Q = L_P$, after decomposing $\R = \bigcup_{j=1}^{C_{N,d}}$, for $(t_0,\dots,t_{2D-1}) \in I_j^d$, we have the lower bound
$$
|\det\bigl(Q'(t_0),\dots,Q'(t_{2D-1})\bigr)| \gtrsim \prod_{j=0}^{2D-1} |L_P(t_j)|^{\frac1d} \prod_{j<k}|t_j-t_k| .
$$
Thus the polynomial $J$ in \eqref{E:dif1} obeys
\begin{equation} \label{E:lb J}
|J(t_0,\dots,t_{2D-1})| \gtrsim \prod_{j=0}^{2D-1} |L_P(t_j)|^{\frac1d}.
\end{equation}

Now we fix $D+1 \leq j \leq 2D-1$ and consider a single derivative from \eqref{E:DQ},
$$
\partial_j |_{t_j = t_{j-D}} \det\bigl(Q'(t_0),\ldots,Q'(t_{2D-1})\bigr).
$$
By \eqref{E:dif1} and the product rule, this is a sum of $\frac{d(d-1)}2 + 1$ terms, one for each of the linear factors in \eqref{E:dif1} and an additional one for $J$.  But it is clear that the only one of these terms that is nonzero is the one in which $\partial_j$ eliminates the $(t_j-t_{j-D})$ factor before the evaluation.  Hence the quantity in \eqref{E:DQ} is equal to
\begin{align*}
\pm \bigl\{\prod_{i=1}^{D-1}(t_D-t_i)^2(t_i-t_0)^2 \prod_{j=i+1}^{D-1}(t_j-t_i)^4\bigr\} (t_D-t_0) J (t_0,\ldots,t_D,t_1,\ldots,t_{D-1}).
\end{align*}
Using this together with \eqref{E:lb J} and \eqref{E:DPsi even}, the proof of Proposition~\ref{P:jacobian bounds} is complete.
 \end{proof}

The lower bounds in Proposition~\ref{P:jacobian bounds} together with the invariances in Lemma~\ref{L:invariance} allow us to make some reductions before we attempt to prove Theorem~\ref{T:main}.

\begin{lemma} \label{L:reductions}
In proving Theorem~\ref{T:main}, it suffices to consider the truncated operator $\tilde X_{\theta}$ given by
\begin{equation} \label{E:new X}
\tilde X_{\theta} f(t,y) = \int_{\R} f(s,y+sP(t))|t|^{\frac{2K\theta}{(d+2)(d-1)}}\chi_I(t)\, ds, \qquad I \subset [0,1],
\end{equation}
where $P$ is a polynomial of degree $N$ such that for $t \in I$ we have 
$$
|L_P(t)| \sim |t|^K,
$$
with $K \in [0,C_{d,N}]$ an integer.  We may further assume that on $\R \times I \times \R \times \cdots$, either \eqref{E:Psi even} or \eqref{E:Psi odd} holds, and that on $I \times \R \times I \times \cdots$, either \eqref{E:Phi even} or \eqref{E:Phi odd} holds, depending on whether $d$ is even or odd.
\end{lemma}

\begin{proof}  %We give the details in the case when $0 < \theta < 1$, the case when $\theta = 0$ being trivial (cf.\ Proposition~\ref{P:L1 to LinftyL1}), and the case when $\theta=1$ being similar to the non-endpoint cases.  
Obviously, in proving \eqref{E:strong p theta bound}, we may assume that $f \geq 0$.  By the triangle inequality, it suffices to bound each of the operators $X_{\theta}^{(j)}$ given by
$$
X_{\theta}^{(j)}f(t,y) = \int_{\R} f(s,y+sP(t))|L_P(t)|^{\frac{2\theta}{(d+2)(d-1)}}\chi_{I_j}(t)\, ds,
$$
with $I_j$ one of the intervals in the decomposition in Proposition~\ref{P:jacobian bounds}.  By the monotone convergence theorem (since $f \geq 0$), in bounding $X_{\theta}^{(j)}$, we may assume that $I_j$ is a bounded interval.  Next, by reparametrizing (linearly) in $t$ and applying Lemma~\ref{L:invariance}, we may assume that $I_j \subset [0,1]$ and that $b_j = 0$.  Multiplying $P$ by a constant and using Lemma~\ref{L:invariance} again, we may assume that $A_j = 1$.  Finally, since $f \geq 0$ and $|L_P(t)| \sim |t|^K$ (after all of our reductions), we may replace the weight $|L_P(t)|^{\frac{2\theta}{(d+2)(d-1)}}$ with $|t|^{\frac{2K\theta}{(d+2)(d-1)}}$.  This completes the proof.
\end{proof}

As it suffices to establish bounds for $\tilde X_{\theta}$, we will work with these operators from here forward and drop the $\tilde{\,}\,$'s from our notation.  We note that under this reduction, the adjoint of $X_{\theta}$ is given by
\begin{equation} \label{E:X*}
X_{\theta}^* g(s,x) = \int_{I} g(t,x-sP(t)) |t|^{\frac{2K\theta}{(d+2)(d-1)}}\, dt .
\end{equation}

%%%%%%%%%%%%%%%%%%%%%%%%%%%%%%%%%%%%%%%%%%%%%%%%%
%%%%%%%%%%%%%%%%%%%%%%%%%%%%%%%%%%%%%%%%%%%%%%%%%

\section{The restricted weak-type bounds} \label{S:RWT}

%%%%%%%%%%%%%%%%%%%%%%%%%%%%%%%%%%%%%%%%%%%%%%%%%
%%%%%%%%%%%%%%%%%%%%%%%%%%%%%%%%%%%%%%%%%%%%%%%%%

The main goal of this section is to prove a restricted weak-type version of Theorem~\ref{T:main}.  Our proof is similar to the proof of the restricted weak-type bound at the endpoint $\theta = 1$ given in \cite{CE1}, but we must make some modifications to deal with the differences in the operators.

We recall the quantity 
$$
\theta_0 = \tfrac{(d+2)(d-1)}{d(d+1)},
$$
which is the unique value of $\theta$ such that $q_\theta = r_\theta$.

\begin{proposition} \label{P:rwt}  Let $d \geq 3$ and let $P:\R \to \R^{d-1}$ be a polynomial of degree $N$.  Then for $\theta_0 \leq \theta \leq 1$, $X_{\theta}$ satisfies the restricted  weak-type bound
\begin{equation} \label{E:rwt}
\langle X_{\theta} \chi_E,\chi_F \rangle \lesssim |E|^{\frac1{p_{\theta}}} \|\chi_F\|_{L^{q_{\theta}'}(L^{r_{\theta}'})}
\end{equation}
for all measurable $E,F \subset \R^d$.
\end{proposition}

Before beginning the proof of Proposition~\ref{P:rwt}, we make a minor reduction.

\begin{lemma} \label{L:reduce rwt}  
It suffices to prove Proposition~\ref{P:rwt} under the additional hypothesis that there exists a constant $\beta$ such that $\beta \leq X^*_{\theta}\chi_F(s,x) \leq 2\beta$ for each $(s,x) \in E$.  
\end{lemma}

\begin{proof}[Proof of Lemma~\ref{L:reduce rwt}]  Assume that the proposition has been proven under the additional hypothesis given in the lemma and let $E,F \subset \R^d$ be measurable sets.  By the monotone convergence theorem, we may assume that $E,F$ are bounded sets, and we may of course assume that $E,F$ have positive measures.

Let $\beta = \frac{\langle X_\theta \chi_E,\chi_F\rangle}{|E|}$.  For $n \in \Z$, let
$$
E^n = \{(s,x) \in E : 2^{n-1} \beta < X_\theta^*\chi_F(s,x) \leq 2^n \beta\}.
$$
Then standard arguments show that
$$
\langle X_\theta \chi_{\bigcup_{n=-\infty}^{-1} E^n}, \chi_F\rangle \leq \tfrac12 \langle X_\theta \chi_E,\chi_F\rangle,
$$
which implies that 
$$
\tfrac12 \langle X_\theta \chi_E,\chi_F\rangle \leq \sum_{n=0}^{\infty} \langle X_\theta \chi_{E^n},\chi_F\rangle.
$$
Furthermore,
$$
\beta|E| = \langle X_\theta\chi_E,\chi_F\rangle \geq \langle X_\theta \chi_{E^n},\chi_F\rangle \geq 2^{n-1}\beta|E^n|,
$$
so $|E^n| \lesssim 2^{-n}|E|$.  Thus by our assumption, we have
\begin{align*}
\langle X_\theta \chi_E,\chi_F\rangle \lesssim \sum_{n=0}^{\infty} \langle X_\theta \chi_{E^n},\chi_F\rangle &\lesssim \sum_{n=0}^{\infty} 2^{-\frac{n}{p_\theta}}|E|^{\frac1{p_{\theta}}}\|\chi_F\|_{L^{q_\theta'}(L^{r_\theta'})} \lesssim |E|^{\frac1{p_\theta}}\|\chi_F\|_{L^{q_\theta'}(L^{r_\theta'})},
\end{align*}
i.e.\ \eqref{E:rwt} holds.  This completes the proof of the lemma.
\end{proof}

We record two more lemmas before proceeding to the main part of the proof of Proposition~\ref{P:rwt}.

\begin{lemma}[\cite{CE1}] \label{L:mixed lb}
Let $F \subset \R^d$.  Then
$$
\|\chi_F\|_{L^{q_\theta'}(L^{r_\theta'})} \geq |F|^{\frac1{r_\theta'}}|\Pi(F)|^{\frac1{q_\theta'}-\frac1{r_\theta'}} \qquad \theta_0 \leq \theta \leq 1,
$$
where $\Pi:\R^d \to \R$ denotes the projection $\Pi(t,y) = t$.
\end{lemma}

\begin{proof}  For $\theta$ in the specified range, we have $r_\theta \geq q_\theta$, so $r_\theta' \leq q_\theta'$.  Thus by H\"older's inequality, we have
\begin{align*}
|F| &= \int_{\Pi(F)}\int_{\R^{d-1}} \chi_F(t,y)\, dy\, dt \leq |\Pi(F)|^{1-\frac{r_\theta'}{q_\theta'}} \left(\int_\R \left(\int_{\R^{d-1}} \chi_F(t,y)\, dy\right)^{\frac{q_\theta'}{r_\theta'}}\, dt \right)^{\frac{r_\theta'}{q_\theta'}}.
\end{align*}
\end{proof}

Next is a variant of a lemma from \cite{CE1}.  For $0 \leq \theta \leq 1$, we let $\mu_{\theta}$ denote the measure satisfying
\begin{equation} \label{E:def mu theta}
\int_\R f(t)\, d\mu_{\theta}(t) = \int_I f(t)\, |t|^{\frac{2K\theta}{(d+2)(d-1)}}\, dt.
\end{equation}

\begin{lemma} \label{L:stop time}  Let $\eps > 0$.  Then there exists a constant $c_{\eps} > 0$ such that for every interval $I_0 \subset [0,1]$ and every Lebesgue measurable $S \subset I_0$, there exists an interval $J \subset I_0$ such that $\mu_\theta(J \cap S) \geq \frac12\mu_\theta(S)$ and such that for every interval $I' \subset J$ with $\mu_\theta(I') = \frac12\mu_\theta(J)$, we have that
$$
\mu_\theta(S \cap (J \setminus I')) \geq c_{\eps} \bigl(\tfrac{\mu_\theta(S)}{\mu_\theta(J)}\bigr)^{\eps}\mu_\theta(S).
$$
\end{lemma}

\begin{proof}
We will use a stopping time argument to find an interval $J$ having measure $\mu_\theta(J) = 2^{m}\mu_\theta(S)$ such that $\mu_\theta(J \cap S) \geq \frac12\mu_\theta(S)$ and such that for any interval $I' \subset J$ with $\mu_\theta(I') = \frac12\mu_\theta(J)$, 
$$
\mu_\theta(S \cap I') < (1-c_{\eps}2^{-m \eps})\mu_\theta(S \cap J).
$$

With $I_0$ as in the statement of the lemma, define $m_0\geq 0$ so that $\mu_\theta(I_0) = 2^{m_0}\mu_\theta(S)$.  Let $c > 0$ be a fixed constant whose value will be determined in a moment.  We argue inductively.  Given $I_j$, if there exists an interval $I' \subset I_j$ with $\mu_\theta(I') = \frac12\mu_\theta(I_j)$ and
$$
\mu_\theta(S \cap I') \geq (1-c2^{\eps(j-m_0)})\mu_\theta(S \cap I_j),
$$
then we let $I_{j+1} = I'$ (for one such interval $I'$) and continue.  Otherwise we stop.

We observe that at the $j$-th stage,
$$
\mu_\theta(S \cap I_j) \geq (1-c2^{\eps(j-1-m_0)}) \cdots (1-c2^{-\eps m_0}) \mu_\theta(S).
$$
In particular, if $c = c_{\eps}$ is taken sufficiently small, then
$$
\mu_\theta(S \cap I_j) \geq \tfrac12 \mu_\theta(S), \qtq{for all $j \leq m_0+2$.}
$$
But since $j > m_0+1$ implies that $\mu_\theta(I_j) < \tfrac12\mu_\theta(S)$, the procedure must stop while $j \leq m_0+1$.  The proof is thus complete.
\end{proof}

We are now ready to prove Proposition~\ref{P:rwt}.

\begin{proof}[Proof of Proposition~\ref{P:rwt}] 
By Lemma~\ref{L:reductions}, our goal is to establish the restricted weak-type estimate \eqref{E:rwt} for $X_{\theta} = \tilde X_{\theta}$ in the reduced form \eqref{E:new X}.  Let $E,F \subset \R^d$ be measurable sets having finite, positive measures.  By Lemma~\ref{L:reduce rwt}, we may assume that 
$$
0 < \beta \leq X_\theta^*\chi_F(s,x) \leq 2\beta < \infty
$$
for each $(s,x) \in E$.  

We define 
\begin{equation} \label{E:I beta deltaK}
I_\beta = [c\beta^{\delta_{K\theta}},1], \qquad \delta_{K\theta} = \bigl(1+\tfrac{2K\theta}{(d+2)(d-1)}\bigr)^{-1}
\end{equation}
 and observe that for $c$ sufficiently small, we have $\mu_\theta(I \setminus I_\beta) \ll \beta$.  (Recall that $I\subseteq[0,1]$ and $\mu_\theta$ is given by \eqref{E:def mu theta}.)

Let $(s,x) \in E$ and observe that
$$
X_\theta^* \chi_F(s,x) = \mu_\theta(\{t \in I : (t,x-sP(t)) \in F\}).
$$
Thus if we define
$$
S_{(s,x)} = \{t \in I_\beta : (t,x-sP(t)) \in F\},
$$
we have $\mu_\theta(S_{(s,x)}) \sim \beta$.  

Let $\eps > 0$ be a small quantity to be determined in a moment.  By Lemma~\ref{L:stop time}, there exists an interval $I_{(s,x)} \subset I_\beta$ such that 
\begin{align*}
&\mu_\theta(I_{(s,x)}) = 2^{m_{(s,x)}}\beta, \qtq{with $m_{(s,x)} \geq -C$ for some integer constant $C$,}\\
&\mu_\theta(I_{(s,x)} \cap S_{(s,x)}) \sim \beta,
\end{align*}
and for any interval $I' \subset I_{(s,x)}$ with $\mu_\theta(I') = \frac12\mu_\theta(I_{(s,x)})$ we have 
$$
\mu_\theta(S_{(s,x)} \cap (I_{(s,x)}\setminus I')) \geq c_\eps 2^{-\eps m_{(s,x)}}\beta.
$$
We partition $E$ as $E = \bigcup_{m=-C}^\infty E^m$, where
$$
E^m = \{(s,x) : m-1 < m_{(s,x)} \leq m\}.
$$

With $m$ fixed, we choose points $T_j$, $0 \leq j \leq M$, satisfying
\begin{gather*}
\inf (I \cap I_{\beta}) = T_0 < T_1 < \cdots < T_M = 1\\
2^{m-1}\beta \leq \mu_\theta([T_{j-1},T_j]) \leq 2^m\beta, \quad j=1,\cdots,M.
\end{gather*}
Let $J^m_j = [T_{j-1},T_j]$, for $1 \leq j \leq M$ and $J^m_0=\emptyset$.  It is easy to see that
\begin{equation} \label{E:length Jmj}
|J^m_j| \sim 2^m \beta T_j^{-\frac{2K\theta}{(d+2)(d-1)}}, \qquad 1 \leq j \leq M.
\end{equation}

For $j=1,\dots,M$, we define
$$
E_j^m = \bigl\{(s,x) \in E^m : I_{(s,x)} \subset J_{j-1}^m \cup J_j^m \cup J_{j+1}^m\bigr\}
$$
and observe that if $(s,x) \in E_j^m$, then 
\begin{equation} \label{E:betamj is beta}
\beta \sim X_\theta^*\chi_F(s,x) \sim \mu_\theta(S_{(s,x)}) \sim X_\theta^*\chi_{F_j^m}(s,x),
\end{equation}
where $F_j^m := \Pi^{-1}(J_{j-1}^m \cup J_j^m \cup J_{j+1}^m) \cap F$.  With $m$, $j$ fixed, we define 
$$
\alpha_j^m = \frac{\langle X_\theta \chi_{E_j^m},\chi_{F_j^m}\rangle}{|F_j^m|}.
$$

We will soon prove the following.

\begin{lemma} \label{L:bound Fmj}
Provided $\theta \ge \theta_0$ and $\eps$ is sufficiently small, depending on $N$ and $d$, there exists $\rho > 0$ such that
\begin{equation} \label{E:bound Fmj}
|F^m_j|^{\frac1{r_\theta'} - \frac1{p_\theta'}} \gtrsim 2^{m\rho}(\alpha^m_j)^{\frac1{p_\theta'}}\beta^{\frac1{p_\theta}}|J^m_j|^{\frac1{r_\theta'} - \frac1{q_\theta'}}.
\end{equation}
The implicit constant and $\rho$ are both independent of $\theta$.
\end{lemma}

Assuming the lemma for now, we complete the proof of Proposition~\ref{P:rwt}.  

Using the definition of $\alpha_j^m$ and the fact that $\frac{\langle X_\theta \chi_{E_j^m},\chi_{F_j^m}\rangle}{|E_j^m|} \sim \beta$ (which follows from \eqref{E:betamj is beta}), we see after a little algebra that, provided $\eps$ is sufficiently small,  we have
\begin{equation} \label{E:bound scriptX}
\langle X_\theta \chi_{E^m_j},\chi_{F^m_j}\rangle \lesssim 2^{-\rho m}\frac{|E^m_j|^{\frac1{p_\theta}}|F^m_j|^{\frac1{r_\theta'}}}{|J^m_j|^{\frac1{r_\theta'} - \frac1{q_\theta'}}} \leq 2^{-\rho m}|E^m_j|^{\frac1{p_\theta}}\|\chi_{F^m_j}\|_{L^{q_\theta'}(L^{r_\theta'})}.
\end{equation}

Using the fact that $E^m = \bigcup_j E^m_j$, the relation \eqref{E:betamj is beta}, the bound \eqref{E:bound scriptX}, and H\"older's inequality (recall that $p_\theta \leq q_\theta$ for $\theta \geq \frac{d^2+d-2}{d^2+d}$), we obtain
\begin{align*}
\langle X_\theta\chi_{E^m},\chi_F\rangle &\leq \sum_j \langle X_\theta \chi_{E^m_j},\chi_F\rangle \sim \sum_j \langle X_\theta\chi_{E^m_j},\chi_{F^m_j}\rangle \\
& \lesssim 2^{-\rho m}\sum_j |E^m_j|^{\frac1{p_\theta}} \|\chi_{F^m_j}\|_{L^{q_\theta'}(L^{r_\theta'})}\\
&\leq 2^{-\rho m} \bigl(\sum_j |E^m_j|\bigr)^{\frac1{p_\theta}} \sup_j \|\chi_{F^m_j}\|_{L^{q_\theta'}(L^{r_\theta'})}^{1-\frac{q_\theta'}{p_\theta'}} \bigl(\sum_j \|\chi_{F^m_j}\|_{L^{q_\theta'}(L^{r_\theta'})}^{q_\theta'}\bigr)^{\frac1{p_\theta'}}.
\end{align*}
To bound this last term, we observe that a point $t \in I$ can lie in $\Pi(F^m_j)$ for at most three values of $j$, and similarly, a point $(s,x) \in E^m$ can lie in $E^m_j$ for at most three values of $j$.  Thus
\begin{equation} \label{E:bounded overlap}
\sum_j |E^m_j| \lesssim |E| \qquad \sum_j \|\chi_{F^m_j}\|_{L^{q_\theta'}(L^{r_\theta'})}^{q_\theta'} \lesssim \|\chi_F\|_{L^{q_\theta'}(L^{r_\theta'})}^{q_\theta'},
\end{equation}
which by the computation above implies that
$$
\langle X_\theta \chi_{E^m},\chi_F\rangle \lesssim 2^{-\rho m} |E^m|^{\frac1{p_\theta}}\|\chi_F\|_{L^{q_\theta'}(L^{r_\theta'})}.
$$

Summing on $m$, we obtain \eqref{E:rwt}, and the proposition is proved.
\end{proof}

The proof of the key estimate in Lemma~\ref{L:bound Fmj} will be by the method of refinements.  Similar arguments have already appeared in print, but there are some differences that arise here, and so we give a complete proof.

We begin by recalling the iterated mappings $\Phi_{(s_0,x_0)}^k, \Psi_{(t_0,y_0)}^k : \R^k \to \R^d$ given in (\ref{E:def Phi even}-\ref{E:def Psi odd}).  

The following lemmas will reduce the proof of Lemma~\ref{L:bound Fmj} to a computation.

\begin{lemma}\label{L:mixed even}
If $d=2D \geq 4$ is even, then there exist a point $(t_0,y_0) \in F_j^m$ and a set $\Omega_d \subset \R^d$ with $\Psi^d_{(t_0,y_0)}(\Omega_d) \subset F_j^m$ and
\begin{align}  \label{E:omegad volume even} 
\int_{\Omega_d}|t_0t_D|^{\frac{2K\theta}{(d+2)(d-1)}}\bigl(\prod_{k=1}^{D-1}|t_k|^{\frac{4K\theta}{(d+2)(d-1)}}\bigr)\,dt_D\,\cdots\, ds_1 \gtrsim 2^{-\eps m}(\alpha_j^m)^D \beta^D.
\end{align}
Furthermore, $(s_1,t_1,\ldots,s_D,t_D) \in \Omega_d$ implies that $t_i \in J^m_{j-1} \cup J_j^m \cup J^m_{j+1}$ for $0 \leq i \leq D$ and
\begin{gather}
\label{E:si-si-1 even}
|s_i-s_{i-1}| \gtrsim \frac{\alpha_j^m}{|t_{i-1}|^{\frac{2K\theta}{(d+2)(d-1)}}}, \qquad 2 \leq i \leq D,\\
\label{E:tl-ti even}
|t_l-t_i| \gtrsim \beta|t_i t_l|^{-\frac{K\theta}{(d+2)(d-1)}}, \qquad 0 \leq i < l \leq D-1,\\
\label{E:tD-t1 even}
|t_D-t_1| \gtrsim 2^{\frac12(1+\delta_{K\theta})m} \beta |t_D t_1|^{-\frac{K\theta}{(d+2)(d-1)}},\\
\label{E:tD-ti even}
|t_D-t_i| \gtrsim 2^{-\eps m}\beta |t_Dt_i|^{-\frac{K\theta}{(d+2)(d-1)}}, \qquad i=0,2,\ldots,D-1.
\end{gather}
where $\delta_{K\theta}$ is as in \eqref{E:I beta deltaK}.
\end{lemma}

\begin{lemma} \label{L:mixed odd}
If $d=2D+1\geq 3$ is odd, then there exist a point $(s_0,x_0) \in E_j^m$ and a set $\Omega_d \subset \R^d$ with $\Phi^d_{(s_0,x_0)}(\Omega_d) \subset F_j^m$ and
\begin{equation} \label{E:omegad volume odd}
\int_{\Omega_d} \bigl(\prod_{k=1}^D |t_k|^{\frac{4K\theta}{(d+2)(d-1)}}\bigr)|t_{D+1}|^{\frac{2K\theta}{(d+2)(d-1)}}\, dt_{D+1\,}ds_D\, \dots \,dt_1 \gtrsim 2^{-\eps m}(\alpha_j^m)^D \beta^{D+1}.
\end{equation}
Furthermore, $(t_1,s_1,\dots,t_{D+1}) \in \Omega_d$ implies that $t_i \in J^m_{j-1} \cup J_j^m \cup J^m_{j+1}$, for $1 \leq i \leq D+1$ and
\begin{gather}
\label{E:si-si-1 odd}
|s_i-s_{i-1}| \gtrsim \frac{\alpha_j^m}{|t_i|^{\frac{2K\theta}{(d+2)(d-1)}}}, \qquad 1 \leq i \leq D, \\
\label{E:tl-ti odd}
|t_l-t_i| \gtrsim \beta|t_i t_l|^{-\frac{K\theta}{(d+2)(d-1)}}, \qquad 0 \leq i < l \leq D,\\
\label{E:tD+1-t1 odd}
|t_{D+1}-t_1| \gtrsim 2^{\frac12(1+\delta_{K\theta})m} \beta |t_{D+1}t_1|^{-\frac{K\theta}{(d+2)(d-1)}},\\
\label{E:tD+1-ti odd}
|t_{D+1}-t_i| \gtrsim 2^{-\eps m}\beta |t_{D+1}t_i|^{-\frac{K\theta}{(d+2)(d-1)}}, \qquad 2 \leq i \leq D-1,
\end{gather}
where $\delta_{K\theta}$ is as in \eqref{E:I beta deltaK}.
\end{lemma}

We only give a proof in the even dimensional case, the odd dimensional case being similar.

\begin{proof}[Proof of Lemma~\ref{L:mixed even}] 
We begin by refining $E^m_j,F^m_j$. By arguments which by now have appeared many times in the literature (cf.\ \cite{CCC}), there exist sets
\begin{align*}
&\emptyset \neq F^m_{j,D} \subseteq F^m_{j,D-1} \subseteq \cdots \subseteq F^m_{j,0} = F^m_j \\
&\emptyset \neq E^m_{j,D} \subseteq E^m_{j,D-1} \subseteq \dots \subseteq E^m_{j,0} = E^m_j
\end{align*}
so $(t,y) \in F^m_{j,k}$ implies that $X_\theta \chi_{E^m_{j,k-1}}(t,y) \gtrsim \alpha_j^m$ and $(s,x) \in E^m_{j,k}$ implies that $X_\theta^* \chi_{F^m_{j,k}}(s,x) \gtrsim \frac{\mathcal{X}(E_j^m,F_j^m)}{|E_j^m|}=:\beta_m^j \sim \beta$.  In other words
\begin{align}
\label{E:t,y in Fmjk} 
(t,y) \in F^m_{j,k} &\implies |\{s \in \R : (s,y+sP(t)) \in E^m_{j,k-1}\}| \gtrsim \frac{\alpha^m_j}{t^{\frac{2K\theta}{(d+2)(d-1)}}} \\
\label{E:s,x in Emjk}
(s,x) \in E^m_{j,k} &\implies \mu_{\theta}(\{t \in I : (t,x-sP(t)) \in F^m_{j,k}\}) \gtrsim \beta^m_j.
\end{align}
Furthermore, if $(t,x-sP(t)) \in F^m_{j,k}$, then we in fact have 
$$
t \in \Pi(F^m_{j,k}) \subseteq \Pi(F^m_j) \subseteq J^m_{j-1} \cup J^m_j \cup J^m_{j+1}.
$$

Let $s \in \R$ and $t \in I_{\beta}$. Then it is obvious that
$$
|\{s' \in \R : |s'-s| \ll \frac{\alpha^m_j}{t^{\frac{2K\theta}{(d+2)(d-1)}}}\}| \ll \frac{\alpha^m_j}{t^{\frac{2K\theta}{(d+2)(d-1)}}}.
$$
Next we consider the interval 
$$
J_t = [t-c\frac\beta{t^{\frac{2K\theta}{(d+2)(d-1)}}}, t+c\frac\beta{t^{\frac{2K\theta}{(d+2)(d-1)}}}], 
$$
where $c$ is a small constant satisfying the following.  First, since $t \gtrsim \beta^{\delta_K}$ by \eqref{E:I beta deltaK}, we may choose $c$ sufficiently small that $t \in I_{\beta}$ and $t' \in J_t$ implies $t \sim t'$.  Second, $\mu_\theta(J_t) \sim (\sup J_t)^{\frac{2K\theta}{(d+2)(d-1)}}|J_t|$, and since $\sup J_t \in I_\beta$, making $c$ smaller if necessary, $\mu_{\theta}(J_t) \ll \beta$.  With $c$ fixed, $t' \notin J_t$ implies that 
$$
|t'-t| \gtrsim \frac\beta{(tt')^{\frac{2K\theta}{(d+2)(d-1)}}}.
$$

Now let $(t_0,y_0) \in F^m_{j,D}$.  We define
$$
\Omega_1 = \{s_1 \in \R : \Psi^1_{(t_0,y_0)}(s_1) \in E^m_{j,D-1}\}
$$
and then recursively define
\begin{align*}
\Omega_{2k} &= \bigl\{(s_1,t_1,\cdots,s_k,t_k) \in \Omega_{2k-1} \times \R : \Psi^{2k}_{(t_0,y_0)}(s_1,t_1,\cdots,t_k) \in F^m_{j,D-k}, \\ &\qquad \qquad t_k \notin J_{t_i}, 0 \leq i \leq k-1\bigr\}, \\
\Omega_{2k+1} &= \bigl\{(s_1,t_1,\cdots,t_k,s_{k+1}) \in \Omega_{2k} \times \R : \Psi^{2k+1}_{(t_0,y_0)}(s_1,t_1,\cdots,s_{k+1}) \in E^m_{j,D-k},  \\
&\qquad \qquad |s_{k+1}-s_k| \gtrsim \frac{\alpha^m_j}{t_k^{\frac{2K\theta}{(d+2)(d-1)}}}\bigr\}, 
\end{align*}
for $1 \leq k \leq D-1$.

To define $\Omega_d$, we must argue a little differently.  Let $(s_1,t_1,\cdots,s_D) \in \Omega_{d-1}$ and define $x_D$ so that $(s_D,x_D) = \Psi^{d-1}_{(t_0,y_0)}(s_1,t_1,\dots,s_D)$.  We observe that
$$
\mu_\theta(\{t \in S_{(s_D,x_D)} : |t-t_1| \ll |I_{(s_D,x_D)}|\}) \ll \mu_\theta(I_{(s_D,x_D)}),
$$
so by the way the intervals $I_{(s,x)}$ were defined (cf.\ Lemma~\ref{L:stop time}), we have
$$
\mu_\theta(\{t \in S_{(s_D,x_D)} : |t-t_1| \gtrsim |I_{(s_D,x_D)}|\}) \gtrsim 2^{-\eps m}\beta.
$$
Furthermore,
$$
\mu_\theta(\bigcup_{i=2}^{D-1} [t_i - c\frac{2^{-\eps m}\beta}{t_i^{\frac{2K\theta}{(d+2)(d-1)}}}, t_i + c\frac{2^{-\eps m}\beta}{t_i^{\frac{2K\theta}{(d+2)(d-1)}}}]) \ll 2^{-\eps m}\beta.
$$
Therefore if we define 
\begin{align*}
\Omega_d &= \{(s_1,t_1,\dots,s_D,t_D) \in \Omega_{d-1} \times \R : t_D \in S_{(s_D,x_D)}, \,\,\, |t_D-t_1| \gtrsim |I_{(s_D,x_D)}|, \\ 
&\qquad \qquad \qquad 
\text{and}\,\,\,\, |t_D-t_i| \gtrsim \frac{2^{-\eps m}\beta}{t_i^{\frac{2K\theta}{(d+2)(d-1)}}}, \,\, i=0,2,\ldots,D-1\},
\end{align*}
then $\Omega_d$ satisfies the volume lower bound \eqref{E:omegad volume even}.  We have already shown that \eqref{E:si-si-1 even} and \eqref{E:tl-ti even} hold on $\Omega_d$, and \eqref{E:tD-ti even} is proved in the same way as \eqref{E:tl-ti even}.

We finally turn to \eqref{E:tD-t1 even}.  Let $(s_1,t_1,\dots,t_D) \in \Omega_d$, and define $x_D$ as above.  For each $1 \leq j \leq M$, the requirements $\mu_\theta(J^m_j) \sim 2^m\beta$ and $J^m_j \subset I_\beta$ imply that 
\begin{equation} \label{E:sup gtrsim exp inf}
\inf(J^m_j) \gtrsim 2^{-m\delta_{K\theta}}\sup(J^m_j).
\end{equation}
Indeed, if $\sup (J^m_j) \geq 2 \inf(J^m_j)$, then
$$
\sup(J^m_j)^{\frac{2K\theta}{(d+2)(d-1)}+1} \sim \sup(J^m_j)^{\frac{2K\theta}{(d+2)(d-1)}}|J^m_j| \sim \mu_\theta(J^m_j) \sim 2^m \beta,
$$
from which we deduce that 
$$
\sup(J^m_j) \lesssim (2^m\beta)^{\delta_{K\theta}} \lesssim 2^{m\delta_{K\theta}}\inf(J^m_j),
$$
since $J^m_j \subset I_\beta$.  In fact, a similar argument shows that 
\begin{equation} \label{E:sup sim inf}
\inf(J^m_j) \sim \sup(J^m_j), \qquad j \geq 2,
\end{equation}
because in this case, 
$$
\inf(J^m_j) \gtrsim \sup(J^m_1) \gtrsim (2^m\beta)^{\delta_{K\theta}}.
$$

By \eqref{E:sup gtrsim exp inf}, since $t_1,t_D \in J^m_{j-1} \cup J^m_j \cup J^m_{j+1}$ and $I_{(s_D,x_D)} \subset J^m_{j-1} \cup J^m_j \cup J^m_{j+1}$, we have
$$
\min\{t_1,t_D\} \gtrsim 2^{-m\delta_{k\theta}}\sup(J^m_{j-1}\cup J^m_j \cup J^m_{j+1}) \gtrsim 2^{-m\delta_{k\theta}} \sup I_{(s_D,x_D)}.
$$
Additionally, since $|t_1-t_D| \gtrsim |I_{(s_D,x_D)}|$ and $t_1,t_D > 0$, we have
$$
\max\{t_1,t_D\} \gtrsim \sup I_{(s_D,x_D)}.
$$
We may therefore conclude that 
\begin{align*}
&|t_1-t_D|(t_1t_D)^{\frac{K\theta}{(d+2)(d-1)}} \gtrsim |I_{(s_D,x_D)}|2^{-m\delta_{K\theta} \frac{K\theta}{(d+2)(d-1)}}(\sup I_{(s_D,x_D)})^{\frac{2K\theta}{(d+2)(d-1)}} \\
&\qquad \qquad 
\sim 2^{-m\delta_{K\theta}\frac{K\theta}{(d+2)(d-1)}}\mu_\theta(I_{(s_D,x_D)}) \sim 2^{m(1-\delta_{K\theta}\frac{K\theta}{(d+2)(d-1)})}\beta = 2^{\frac12(1 + \delta_{K\theta})m}\beta.
\end{align*}
This completes the proof of \eqref{E:tD-t1 even} and thus of the lemma.
\end{proof}

Now we are ready to prove Lemma~\ref{L:bound Fmj}.

\begin{proof}[Proof of Lemma~\ref{L:bound Fmj}]
We give the details of the necessary computation when $d=2D$.  Since $\Psi^d_{(t_0,y_0)}$ is a polynomial, a standard application of Bezout's theorem gives
$$
|F^m_j| \geq |\Psi^d_{(t_0,y_0)}(\Omega_d)| \gtrsim \int_{\Omega_d} |\det D\Psi^d_{(t_0,y_0)}(s_1,t_1,\dots,s_D,t_D)|\, dt_D\, \dots\, ds_1.
$$
Thus by Proposition~\ref{P:jacobian bounds}, the lower bounds (\ref{E:si-si-1 even}-\ref{E:tD-ti even}), and a bit of arithmetic,
\begin{align*}
|F^m_j| &\gtrsim \int_{\Omega_d} \prod_{i=1}^{D-1}\bigl\{|s_{i+1}-s_i||t_i|^{\frac{K}D}\prod_{\stackrel{0 \leq j \leq D}{j\neq i}}|t_j-t_i|^2\bigr\} |t_0t_D|^{\frac{K}{2D}}|t_D-t_0|\, dt_D\, \dots\, ds_1\\
&\gtrsim \int_{\Omega_d} (\alpha^m_j)^{D-1}\beta^{2D(D-1)+1}2^{m(1+\delta_{K\theta} - \eps(d-2))}|t_0t_D|^{\frac{K}d - \frac{K\theta}{d+2}} \prod_{i=1}^{D-1}|t_i|^{\frac{2K}d - \frac{2K\theta}{d+2}}.
\end{align*}
To take advantage of \eqref{E:omegad volume even}, we use the fact that $t_i \leq T^m_j$ (by \eqref{E:sup sim inf}, since $t_i \in J^m_{j-1} \cup J^m_j \cup J^m_{j+1}$), and obtain
\begin{align*}
|F^m_j| &\gtrsim (\alpha^m_j)^{D-1} \beta^{2D(D-1)+1} 2^{m(1+\delta_{K\theta} - \eps(d-2))}(T^m_j)^{-K(\frac{d(d+1)\theta}{(d+2)(d-1)} - 1)}\\ & \qquad \qquad  \times\int_{\Omega_d}|t_0t_D|^{\frac{2K\theta}{(d+2)(d-1)}} \prod_{i=1}^{D-1}|t_i|^{\frac{4K\theta}{(d+2)(d-1)}}\, dt_D \, \dots \, ds_1 \\
&\gtrsim (\alpha^m_j)^{d-1} \beta^{\frac{d^2-d+2}2} 2^{m(1+\delta_{K\theta} - \eps(d-3))}(T^m_j)^{-K(\frac{d(d+1)\theta}{(d+2)(d-1)} - 1)}.
\end{align*}
Now we use the fact that $|J^m_j|(T^m_j)^{\frac{2K\theta}{(d+2)(d-1)}} \sim 2^m \beta$ and a bit of arithmetic to see that if $\eps$ is sufficiently small, then
\begin{align*}
|F^m_j| &\gtrsim (\alpha^m_j)^{d-1} \beta^{\frac{(d+2)(d-1)}{2\theta} - (d-1)} 2^{m(\delta_{K\theta} + \frac{(d+2)(d-1)}2 (\frac1\theta - 1)-\eps(d-2))} |J^m_j|^{\frac{d(d+1)}2 - \frac{(d+2)(d-1)}{2\theta}} \\
&= 2^{m\rho} \alpha^{\frac1{p_\theta'}(\frac1{r_\theta'}-\frac1{p_\theta'})^{-1}} \beta^{\frac1{p_\theta}(\frac1{r_\theta'} - \frac1{p_\theta'})^{-1}}|J^m_j|^{(\frac1{r_\theta'}-\frac1{q_\theta'})(\frac1{r_\theta'}-\frac1{p_\theta'})^{-1}}.
\end{align*}
This completes the proof.
\end{proof}

%%%%%%%%%%%%%%%%%%%%%%%%%%%%%%%%%%%%%%%%%%%%%%%%%
%%%%%%%%%%%%%%%%%%%%%%%%%%%%%%%%%%%%%%%%%%%%%%%%%

\section{Interpolation} \label{S:Interpolation}

%%%%%%%%%%%%%%%%%%%%%%%%%%%%%%%%%%%%%%%%%%%%%%%%%
%%%%%%%%%%%%%%%%%%%%%%%%%%%%%%%%%%%%%%%%%%%%%%%%%

If we were working with unweighted, non-mixed $L^p \to L^q$ bounds with $(p^{-1},q^{-1})$ lying on a line segment, then strong-type bounds away from the endpoints would follow from the restricted weak type bounds at the endpoints via Marcinkiewicz interpolation.  The mixed norms seem to present particular difficulties.  Although there are cases (such as in \cite{Milman81}) where interpolation of multiple inequalities has been used to deduce strong type mixed norm estimates from restricted weak type bounds, the authors are unaware of an analogue of Marcinkiewicz interpolation that applies in the current situation, where the triples $(p^{-1},q^{-1},r^{-1})$ under consideration all lie on a single line segment.  (A closely related issue is discussed in \cite[Sections 4 and 5]{KeelTao}.)

Despite these difficulties, in this section we will use interpolation to obtain an improvement, albeit a non-optmal one, for non-endpoint values of $\theta$.  We will use this interpolation step later on in the proof of Theorem~\ref{T:main}.  

We continue with the simplifications in Lemma~\ref{L:reductions} in place.  Let $k \geq 0$, and define an operator $X^{-k}_\theta$ by 
\begin{equation} \label{E:Xk}
\begin{aligned}
X_\theta^{-k} f(t,y) &= X_\theta f(t,y) \chi_{(2^{-k-1},2^{-k}]}(t) \\&= \int_\R f(s,y+sP(t))|t|^{\frac{2K\theta}{(d+2)(d-1)}}\chi_{I \cap (2^{-k-1},2^{-k}]}(t)\, ds.
\end{aligned}
\end{equation}

\begin{proposition} \label{P:interpolated Xray}
For $\theta_0 < \theta < 1$ and $k \geq 0$, the operator $X^{-k}_\theta$ defined by \eqref{E:Xk} is a bounded operator from $L^{p_\theta,1}$ to $L^{q_\theta,\infty}(L^{r_\theta})$.  
\end{proposition}

This proposition follows from Proposition~\ref{P:rwt} and by applying the following lemma to $X_0^{-k}$ with 
$$
(s_0,u_0,v_0) = (p_{\theta_0},q_{\theta_0},r_{\theta_0}), \qquad (s_1,u_1,v_1) = (p_1,q_1,r_1)
$$
and $M_0 = 2^{\frac{k(2K\theta_0)}{(d+2)(d-1)}}$, $M_1 = 2^{\frac{k(2K)}{(d+2)(d-1)}}$.

\begin{lemma} \label{L:interpolation}  Let $T$ be a positive operator satisfying the restricted weak-type bounds
$$
\langle T\chi_E,\chi_F \rangle \le M_j |E|^{\frac1{s_j}}\|\chi_F\|_{L^{u_j'}(L^{v_j'})}, \qquad j=0,1,
$$
for all measurable sets $E,F \subset \R^d$, where $1 < s_j,u_j,v_j < \infty$, $j=0,1$ and $v_0 < v_1$.  Then $T$ satisfies the mixed Lorentz space bound
\begin{equation} \label{E:interpolated}
|\langle T f, g \rangle| \lesssim M_0^{1-\theta}M_1^{\theta}\|f\|_{L^{s,1}} \|g\|_{L^{u',1}(L^{v'})},
\end{equation}
for any measurable functions $f,g$ and any triple $(s,u,v)$ satisfying
\begin{equation} \label{E:suv}
(s^{-1},u^{-1},v^{-1}) = (1-\theta)(s_0^{-1},u_0^{-1},v_0^{-1}) + \theta(s_1^{-1},u_1^{-1},v_1^{-1}), \quad 0 < \theta < 1.
\end{equation}
\end{lemma}

The proof this lemma is along standard lines, but we give the details to facilitate our understanding of the quasi-extremizers for \eqref{E:interpolated}.

\begin{proof}  Before beginning, we recall that if $f$ is a measurable function with $|f| \sim \sum_j 2^j \chi_{E_j}$, where the $E_j$ are pairwise disjoint measurable sets, then 
$$
\|f\|_{L^{s,1}} \sim \sum_j 2^j |E_j|^{\frac1s}.
$$
Thus by linearity and positivity of $T$, it suffices to prove the bound \eqref{E:interpolated} when $f = \chi_E$ for some measurable set $E$ and $g = \sum_j 2^j \chi_{F_j}$ for pairwise disjoint measurable sets $F_j$ satisfying 
\begin{equation} \label{E:Lv' const}
\|\sum_j 2^j \chi_{F_j}(t,\cdot)\|_{L^{v'}} \sim A, \qtq{for all $t \in F^{\flat} := \Pi\bigl(\bigcup_j F_j\bigr)$,}
\end{equation}
for some $A > 0$.  We note that our assumption implies that
$$
\|g\|_{L^{u',1}(L^{v'})} \sim \|g\|_{L^{u'}(L^{v'})} \sim A|F^\flat|^{\frac1{u'}}.
$$

To prove \eqref{E:interpolated} with $f$ and $g$ as above, we let $n$ be an integer (whose value will be determined in a moment) and decompose
$$
\langle T \chi_E, \sum_j 2^j \chi_{F_j} \rangle = \langle T \chi_E, \sum_{j=-\infty}^n 2^j F_j \rangle + \langle T \chi_E, \sum_{j=n+1}^{\infty} 2^j \chi_{F_j} \rangle = \Sigma_0 + \Sigma_1.
$$

First we bound $\Sigma_0$.  By assumption, we have
$$
\Sigma_0 \leq M_0 |E|^{\frac1{s_0}} \sum_{j=-\infty}^n 2^j \|\chi_{F_j}\|_{L^{u_0'}(L^{v_0'})}.
$$
With $-\infty < j \leq n$ fixed, we have by \eqref{E:Lv' const} that
\begin{align} \label{E:sigma0 Fj} 
\begin{aligned}
2^j \|\chi_{F_j}\|_{L^{u_0'}(L^{v_0'})} &= \bigl(\int 2^{ju_0'(1-\frac{v'}{v_0'})}(2^{jv'}|\Pi^{-1}(t) \cap F_j|)^{\frac{u_0'}{v_0'}}\, dt \bigr)^{\frac1{u_0'}} \\
&\lesssim 2^{j(1-\frac{v'}{v_0'})}A^{\frac{v'}{v_0'}} |F^{\flat}|^{\frac1{u_0'}} \sim 2^{j(1-\frac{v'}{v_0'})}A^{\frac{v'}{v_0'} - \frac{u'}{u_0'}}\|g\|_{L^{u'}(L^{v'})}^{\frac{u'}{u_0'}}.
\end{aligned}
\end{align}
Since $v_0 < v_1$, we have $v' < v_0'$, and we may conclude that
\begin{equation} \label{E:sigma0}
\Sigma_0 \lesssim M_0 2^{n(1-\frac{v'}{v_0'})}A^{\frac{v'}{v_0'} - \frac{u'}{u_0'}} |E|^{\frac1{s_0}} \|g\|_{L^{u'}(L^{v'})}^{\frac{u'}{u_0'}}.
\end{equation}

Arguing similarly, using the fact that $v_1' < v'$, we have
\begin{align}\label{E:sigma1}
\begin{aligned}
\Sigma_1 &\leq M_1 |E|^{\frac1{s_1}} \sum_{j=n+1}^\infty 2^j \|\chi_{F_j}\|_{L^{u_1'}(L^{v_1'})} \\& \lesssim M_1 2^{n(1-\frac{v'}{v_1'})}A^{\frac{v'}{v_1'} - \frac{u'}{u_1'}} |E|^{\frac1{s_1}} \|g\|_{L^{u'}(L^{v'})}^{\frac{u'}{u_1'}}.
\end{aligned}
\end{align}

The final step is to find $n$ satisfying
\begin{equation} \label{E:N}
\begin{gathered}
 M_0^{1-\theta}M_1^\theta |E|^{\frac1s} \|g\|_{L^{u'}(L^{v'})} \sim M_0 2^{n(1-\frac{v'}{v_0'})}A^{\frac{v'}{v_0'} - \frac{u'}{u_0'}} |E|^{\frac1{s_0}} \|g\|_{L^{u'}(L^{v'})}^{\frac{u'}{u_0'}} \\
\sim M_1 2^{n(1-\frac{v'}{v_1'})}A^{\frac{v'}{v_1'} - \frac{u'}{u_1'}}|E|^{\frac1{s_1}} \|g\|_{L^{u'}(L^{v'})}^{\frac{u'}{u_1'}} .
\end{gathered}
\end{equation}
This is possible because of the identity
\begin{align*}
&\bigl( \bigl(\tfrac{M_1}{M_0}\bigr)^{\theta} A^{\frac{u'}{u_0'} - \frac{v'}{v_0'}}|E|^{\frac1s - \frac1{s_0}}\|g\|_{L^{u'}(L^{v'})}^{1-\frac{u'}{u_0'}}\bigr)^{\frac1{1-\frac{v'}{v_0'}}}\\
&\qquad=\bigl( \bigl(\tfrac{M_0}{M_1}\bigr)^{1-\theta} A^{\frac{u'}{u_1'} - \frac{v'}{v_1'}}|E|^{\frac1s - \frac1{s_1}}\|g\|_{L^{u'}(L^{v'})}^{1-\frac{u'}{u_1'}}\bigr)^{\frac1{1-\frac{v'}{v_1'}}}.
 \end{align*}
We leave the computations to the reader.
\end{proof}

Now by paying more careful attention to the losses in the above lemma, we can partially characterize the quasi-extremizers of \eqref{E:interpolated}.  

\begin{lemma} \label{L:quasiex interp}
Let $T$ be an operator satisfying the hypotheses of Lemma~\ref{L:interpolation} and let $(s,u,v)$ be as in \eqref{E:suv}. Let $E$ be a measurable set, $g = \sum_j 2^j \chi_{F_j}$ with the $F_j$ pairwise disjoint sets satisfying \eqref{E:Lv' const}, and assume that $\langle T\chi_E,g \rangle \geq \eps M_0^{1-\theta}M_1^\theta  |E|^{\frac1s}\|g\|_{L^{u'}(L^{v'})}$ for some $0 < \eps \leq \frac12$. 

Then there exists a subset $\scriptJ \subset \Z$ with cardinality $\#\scriptJ \lesssim \log(1+\eps^{-1})$ such that 
$$
\langle T\chi_E , \sum_{j \in \scriptJ} 2^j \chi_{F_j} \rangle \gtrsim \langle T\chi_E,g \rangle,
$$
and such that $j \in \scriptJ$ and $i \in \{0,1\}$ imply that
\begin{equation}\label{E:in J qex}
\langle T \chi_E, \chi_{F_j} \rangle \gtrsim \eps^C M_i |E|^{\frac1{s_i}} \|\chi_{F_j} \|_{L^{u_i'}(L^{v_i'})} 
\end{equation}
and
\begin{equation}
\label{E:in J Fj big}
\begin{aligned}
\eps^C \bigl(\tfrac{M_1}{M_0}\bigr)^{a_i}A^b|E|^{c_i}\|g\|_{L^{u'}(L^{v'})}^{u'd_i} &\lesssim \|\chi_{F_j}\|_{L^{u_i'}(L^{v_i'})} \\
&  \lesssim \eps^{-C}  \bigl(\tfrac{M_1}{M_0}\bigr)^{a_i}A^b|E|^{c_i}\|g\|_{L^{u'}(L^{v'})}^{u'd_i},
\end{aligned}
\end{equation}
where
\begin{equation} \label{E:abcd}
\begin{aligned}
&a_i = \frac1{v_i'(\frac1{v_0'} - \frac1{v_1'})}, \qquad b = \frac{u'(\frac1{v_0'u_1'} - \frac1{u_0'v_1'})}{\frac1{v_1'} - \frac1{v_0'}}, \\
&\qquad c_i = -\frac{\frac1{s_1} - \frac1{s_0}}{v_i'(\frac1{v_1'} - \frac1{v_0'})}, \qquad d_i = \frac1{u_i'} - \frac{\frac1{u_1'} - \frac1{u_0'}}{v_i'(\frac1{v_1'} - \frac1{v_0'})}.
\end{aligned}
\end{equation}
The exponent $C$ and the implicit constants are allowed to depend on $\theta$ and the exponents $s_0,s_1,u_0,u_1,v_0,v_1$.
\end{lemma}

\begin{proof}  Let $n$ be as in \eqref{E:N}.  Then arguing as in \eqref{E:sigma0 Fj}, \eqref{E:sigma0}, and \eqref{E:sigma1}, we have 
\begin{equation} \label{E:j big}
 \sum_{|j-n| >C\log(1+\eps^{-1})}\langle T\chi_E, 2^j \chi_{F_j} \rangle \ll \eps M_0^{1-\theta}M_1^\theta |E|^{\frac1s}\|g\|_{L^{u'}(L^{v'})} \leq \langle T\chi_E,g \rangle.
\end{equation}
Thus if we let 
$$
\scriptJ = \{j \in \Z : |j-n| \leq C \log(1+\eps^{-1}) \text{\: and \:} \langle T\chi_E,2^j \chi_{F_j} \rangle \geq c \eps \langle T\chi_E,g \rangle \},
$$
we have
$$
\langle T\chi_E,\sum_{j \in \scriptJ} 2^j \chi_{F_j} \rangle \gtrsim \langle T\chi_E,g \rangle.
$$

We claim that $\scriptJ$ has all of the properties stated in the conclusion of the lemma.  The cardinality bound is obvious.  For \eqref{E:in J qex} with $i=0$, we use the definition of $n$ and use \eqref{E:N} and then \eqref{E:sigma0 Fj} to see that if $j \in \scriptJ$,
\begin{align*}
\langle T\chi_E,\chi_{F_j} \rangle &\gtrsim 2^{-j} \eps^2M_0^{1-\theta} M_1^\theta  |E|^{\frac1s}\|g\|_{L^{u'}(L^{v'})}\\
& \gtrsim \eps^C 2^{-j\frac{v'}{v_0'}}2^{-n(1-\frac{v'}{v_0'})}M_0^{1-\theta}M_1^\theta |E|^{\frac1s}\|g\|_{L^{u'}(L^{v'})} \\
&\sim \eps^C 2^{-j\frac{v'}{v_0'}}M_0 A^{\frac{v'}{v_0'} - \frac{u'}{u_0'}} |E|^{\frac1{s_0}}\|g\|_{L^{u'}(L^{v'})}^{\frac{u'}{u_0'}} \gtrsim \eps^C M_0 |E|^{\frac1{s_0}} \|\chi_{F_j}\|_{L^{u_0'}(L^{v_0'})}.
\end{align*}
This establishes \eqref{E:in J qex} when $i=0$, and the case $i=1$ may be verified using similar arguments.

We use \eqref{E:in J qex} to prove \eqref{E:in J Fj big}.  We note that $j \in \scriptJ$ implies
\begin{align*}
&\eps M_0^{1-\theta}M_1^\theta |E|^{\frac1s}\|g\|_{L^{u'}(L^{v'})} \leq \langle T\chi_E,g \rangle \\
&\qquad \lesssim 2^j\eps^{-1} \langle T\chi_E,\chi_{F_j} \rangle \lesssim 2^j \eps^{-1}M_i |E|^{\frac1{s_i}} \|\chi_{F_j}\|_{L^{u_i'}(L^{v_i'})}\\
&\eps^C M_i |E|^{\frac1{s_i}}\|\chi_{F_j}\|_{L^{u_i'}(L^{v_i'})} \lesssim \langle T\chi_E,\chi_{F_j} \rangle\\
&\qquad  \leq 2^{-j}\langle T\chi_E,g \rangle \lesssim 2^{-j} M_0^{1-\theta} M_1^\theta |E|^{\frac1s}\|g\|_{L^{u'}(L^{v'})}.
\end{align*}
Taking $i=0$ and rearranging, we have for $j \in \scriptJ$ that 
$$
\begin{aligned}
&\eps^C 2^{-n} \bigl(\tfrac{M_1}{M_0}\bigr)^\theta |E|^{\frac1{s} - \frac1{s_0}}\|g\|_{L^{u'}(L^{v'})} \lesssim \|\chi_{F_j}\|_{L^{u_0'}(L^{v_0'})}\\
&\qquad  \lesssim \eps^{-C} 2^{-n} \bigl(\tfrac{M_1}{M_0}\bigr)^\theta |E|^{\frac1s - \frac1{s_0}}\|g\|_{L^{u'}(L^{v'})}.
\end{aligned}
$$
By \eqref{E:N},
$$
2^{-n} \bigl(\tfrac{M_1}{M_0}\bigr)^\theta |E|^{\frac1{s} - \frac1{s_0}}\|g\|_{L^{u'}(L^{v'})} \sim \bigl(\tfrac{M_1}{M_0}\bigr)^{a_0}A^b|E|^{c_0}\|g\|_{L^{u'}(L^{v'})}^{u'd_0},
$$
so \eqref{E:in J Fj big} holds when $i=0$.  We leave the case $i=1$ to the reader.
\end{proof}

%%%%%%%%%%%%%%%%%%%%%%%%%%%%%%%%%%%%%%%%%%%%%%%%%
%%%%%%%%%%%%%%%%%%%%%%%%%%%%%%%%%%%%%%%%%%%%%%%%%

\section{The strong type bounds} \label{S:Strong type}

%%%%%%%%%%%%%%%%%%%%%%%%%%%%%%%%%%%%%%%%%%%%%%%%%
%%%%%%%%%%%%%%%%%%%%%%%%%%%%%%%%%%%%%%%%%%%%%%%%%

Our main task in this section will be to prove the following, which will complete the proof of Theorem~\ref{T:main}.

\begin{proposition} \label{P:restricted strong type}
Let $\theta_0 < \theta < 1$.  Then $X_\theta$ is a bounded operator from $L^{p_\theta,1}$ to $L^{q_\theta}(L^{r_\theta})$.
\end{proposition}

With $X_\theta$ as in \eqref{E:new X}, we recall from the introduction that this is equivalent to the statement that $X_0$ is a bounded operator from $L^{p_\theta,1}$ to $L^{q_\theta}(L^{r_\theta}; d\gamma^*\lambda)$.  We can thus use real interpolation (cf.\ \cite[1.18.4--6]{Triebel}) and the trivial $L^1 \to L^\infty(L^1)$ bound (Lemma~\ref{L:L1 to LinftyL1}) to obtain the main theorem.

\begin{proof}[Proof of Proposition~\ref{P:restricted strong type}]  
It suffices to prove that $\langle X_\theta \chi_E,g \rangle \lesssim |E|^{\frac1{p_\theta}}$ whenever $E \subset \R^d$ has finite, positive measure and $g = \sum_j 2^j \chi_{F_j}$ with the $F_j$ pairwise disjoint sets and $\|g\|_{L^{q_\theta'}(L^{r_\theta'})} \sim 1$.  

We start by decomposing $g$.  For $(k,l) \in \Z^2$, we define 
\begin{align*}
A_{kl} &:= \{t \in \R : 2^{l-1} < t \leq 2^l, \qtq{and} 2^{k-1} < \|g(t,\cdot)\|_{L^{r_\theta'}} \leq 2^k\}\\
g_{kl} &:= (\chi_{A_{kl}}\circ \Pi) g, \qquad F_{jkl} := \Pi^{-1}(A_{kl}) \cap F_j, \quad j \in \Z.
\end{align*}
We now adapt a strategy, originally used by Christ in \cite{ChQex} and subsequently used in several other articles (cf.\ \cite{DLW, Laghi, StovallJFA}) to prove strong-type bounds for various generalized Radon transforms.  Roughly, we will prove that the $g_{kl}$ interact with almost disjoint pieces of $E$.

Our next step is to decompose the right hand side of the identity
$$
\langle X_\theta \chi_E,g \rangle = \sum_{k,l} \langle X_\theta \chi_E,g_{kl}\rangle = \sum_{k,l} \langle X_\theta^l \chi_E,g_{kl}\rangle,
$$
where we have used the notation in \eqref{E:Xk}.  For $\eps > 0$, we say that $(k,l) \in \scriptK_\eps$ if 
$$
\tfrac\eps2|E|^{\frac1{p_\theta}} \|g_{kl}\|_{L^{q_\theta'}(L^{r_\theta'})} < \langle X_\theta \chi_E,g_{kl}\rangle \leq \eps|E|^{\frac1{p_\theta}} \|g_{kl}\|_{L^{q_\theta'}(L^{r_\theta'})}.
$$
We observe that for each $k,l$, $\|g_{kl}\|_{L^{q_\theta',1}(L^{r_\theta'})} \sim \|g_{kl}\|_{L^{q_\theta'}(L^{r_\theta'})}$, so by Proposition~\ref{P:interpolated Xray}, we have that
$$
\langle X_\theta \chi_E,g\rangle = \sum_{\eps \lesssim 1} \sum_{(k,l) \in \scriptK_\eps} \langle X_\theta \chi_E,g_{kl}\rangle,
$$
where the outer sum is taken only over integer powers of 2.  We decompose further.  For $\eta > 0$, we say that $(k,l) \in \scriptK_{\eps\eta}$ if 
$$
\frac\eta2 < \|g_{kl}\|_{L^{q_\theta'}(L^{r_\theta'})} \leq \eta.
$$

We now record a trivial bound.  Since the $A_{kl}$ are pairwise disjoint, we have
$$
\#\scriptK_{\eps\eta} \sim \eta^{-q_\theta'}\sum_{(k,l) \in \scriptK_{\eps\eta}} \|g_{kl}\|_{L^{q_\theta'}(L^{r_\theta'})}^{q_\theta'} \leq \eta^{-q_\theta'} \|g\|_{L^{q_\theta'}(L^{r_\theta'})}^{q_\theta'} \sim \eta^{-q_\theta'}.
$$
Therefore
\begin{equation} \label{E:trivial bound}
\sum_{(k,l) \in \scriptK_{\eps\eta}}\langle X_\theta \chi_E,g_{kl}\rangle \sim \sum_{(k,l) \in \scriptK_{\eps\eta}} \eps |E|^{\frac1{p_\theta}}\eta \lesssim \eps \eta^{1-q_\theta'}|E|^{\frac1{p_\theta}}.
\end{equation}
Because of the negative power of $\eta$, we need a second bound in addition to \eqref{E:trivial bound}.

We apply Lemmma~\ref{L:quasiex interp} to $\langle X_0 \chi_E,g_{kl}\rangle$ with $A = 2^k$,
$$
(s_0,u_0,v_0) = (p_{\theta_0},q_{\theta_0},r_{\theta_0}), \qquad (s_1,u_1,v_1) = (p_1,q_1,r_1),
$$
and $M_0 = 2^{-\frac{2K\theta_0l}{(d+2)(d-1)}}$, $M_1 = 2^{-\frac{2Kl}{(d+2)(d-1)}}$.  Noting that the quantities in \eqref{E:abcd} are
$$
a_0 = -\frac{(d+2)(d-1)}2, \qquad b = -q'_\theta, \qquad c_0 = \frac{2(d-1)}{d(d+1)}, \qquad d_0 = 1,
$$
there exists a set $\scriptJ_{kl} \subset \Z$, $\#\scriptJ_{kl} \lesssim \log(1+\eps^{-1})$, such that for $j \in \scriptJ_{kl}$, we have (recalling that $q_{\theta_0} = r_{\theta_0}$)
\begin{gather}
\label{E:Fjkl qex}
\langle X_{\theta_0} \chi_E,\chi_{F_{jkl}}\rangle \gtrsim \eps^C |E|^{\frac1{p_{\theta_0}}} |F_{jkl}|^{\frac1{q_{\theta_0}'}}\\
\label{E:Fjkl same size}
\begin{aligned}
&\eps^C2^{l \frac{2K}{d(d+1)}}(2^{-k}\|g_{kl}\|_{L^{q_\theta'}(L^{r_\theta'})})^{q_\theta'}|E|^{\frac{2(d-1)}{d(d+1)}} \lesssim |F_{jkl}|^{\frac1{q_{\theta_0}'}} \\
&\qquad \lesssim \eps^{-C} 2^{l \frac{2K}{d(d+1)}}(2^{-k}\|g_{kl}\|_{L^{q_\theta'}(L^{r_\theta'})})^{q_\theta'}|E|^{\frac{2(d-1)}{d(d+1)}}.
\end{aligned}
\end{gather}

Given $(k,l) \in \scriptK_{\eps\eta}$, we define
\begin{align*}
E_{jkl} &:= \bigl\{(s,x) \in E : X_\theta^*\chi_{F_{jkl}}(s,x) \geq \frac{\langle X_\theta\chi_E,\chi_{F_{jkl}}\rangle}{2|E|}\bigr\}, \quad j \in \scriptJ_{kl}, \\
E_{kl} &:= \bigcup_{j \in \scriptJ_{kl}} E_{jkl}.
\end{align*}
As usual, we have
\begin{equation} \label{E:Ejkl sim E}
\langle X_\theta\chi_E,\chi_{F_{jkl}}\rangle \sim \langle X_\theta\chi_{E_{jkl}},\chi_{F_{jkl}}\rangle.
\end{equation}
In a moment we will prove the following.

\begin{lemma} \label{L:disjoint} We have for each $\eta,\eps$ that
$$
\sum_{(k,l) \in \scriptK_{\eps \eta}}|E_{kl}| \lesssim (\log(1+\eps^{-1}))^3|E|.
$$
\end{lemma}

Assuming the lemma for now and using Lemma \ref{L:interpolation}, we compute
\begin{align} 
\notag
\sum_{(k,l) \in \scriptK_{\eps\eta}} &\langle X_\theta \chi_E,g_{kl}\rangle \sim \sum_{(k,l) \in \scriptK_{\eps\eta}}\sum_{j \in \scriptJ_{kl}}\langle X_\theta \chi_{E_{jkl}},2^j \chi_{F_{jkl}}\rangle \\\notag
& \lesssim \sum_{(k,l) \in \scriptK_{\eps\eta}} \langle X_\theta \chi_{E_{kl}},g_{kl}\rangle 
\lesssim \sum_{(k,l) \in \scriptK_{\eps\eta}} |E_{kl}|^{\frac1{p_\theta}}\|g_{kl}\|_{L^{q_\theta'}(L^{r_\theta'})} \\
\notag
&\leq \bigl(\sum_{(k,l) \in \scriptK_{\eps\eta}} |E_{kl}|\bigr)^{\frac1{p_\theta}} \bigl(\sum_{(k,l) \in \scriptK_{\eps\eta}} \|g_{kl}\|_{L^{q_\theta'}(L^{r_\theta'})}^{q_\theta'}\bigr)^{\frac1{p_\theta'}} \sup_{(k,l) \in \scriptK_{\eps\eta}} \|g_{kl}\|_{L^{q_\theta'}(L^{r_\theta'})}^{1-\frac{q_\theta'}{p_\theta'}} \\
\label{E:harder bound}
&\lesssim \log(1+\eps^{-1})^{\frac3{p_\theta}}|E|^{\frac1{q_\theta}} \|g\|_{L^{q_\theta'}(L^{r_\theta'})}^{\frac{q_\theta'}{p_\theta'}}\eta^{1-\frac{q_\theta'}{p_\theta'}}
\end{align}
where for the second-to-last inequality, we have used H\"older's inequality and the fact that $q_\theta' < p_\theta'$.

We recall that $\|g\|_{L^{q_\theta'}(L^{r_\theta'})} \sim 1$, so interpolating between \eqref{E:trivial bound} and \eqref{E:harder bound}, we obtain
$$
\sum_{(k,l) \in \scriptK_{\eps\eta}} \langle X_\theta \chi_E,g_{kl}\rangle \lesssim \eps^c \eta^d |E|^{\frac1{p_\theta}}
$$
for some constants $c,d > 0$.  Summing over dyadic values of $\eps,\eta$ satisfying $0 < \eps,\eta \lesssim 1$ completes the proof.
\end{proof}

Before we begin the proof of Lemma~\ref{L:disjoint}, we record a useful lower bound.

\begin{lemma}\label{L:improved Laghi}
Let $F,F',H \subset \R^d$ have finite, positive measures.  Assume that $\Pi(F) \subset (2^{l-1},2^l]$, $\Pi(F') \subset (2^{l'-1},2^{l'}]$, 
\begin{equation} \label{E:bigger than beta}
X_{\theta_0}^*\chi_F(s,x) \gtrsim \beta,\quad \tnorm{and} \quad X_{\theta_0}^*\chi_{F'}(s,x) \gtrsim \beta', \quad \tnorm{for} \quad (s,x) \in H.
\end{equation}
Define
$$
\alpha = \frac{\langle X_{\theta_0}\chi_H,\chi_F\rangle}{|F|}.
$$
If $|l-l'| > 1$, then we have
\begin{align*}
|F'| &\gtrsim 2^{a|l-l'|}\alpha^{d-1}(\beta')^{\frac{d+1}2+\lambda}\beta^{\frac{d^2-2d+1}2 - \lambda}, \\
\end{align*}
where $a=\frac{d-1}4$ and $\lambda = \frac{d-1}{4(1+2K/d(d+1))}$.  If $|l-l'| \leq 1$, then we have
\begin{align*}
|F'| &\gtrsim \alpha^{d-1}(\beta')^d \beta^{\frac{d^2-d+2}2-d}.
\end{align*}
\end{lemma}

\begin{proof}[Proof of Lemma~\ref{L:improved Laghi}]  
We give the details when $d=2D+1$, the case when $d=2D$ being similar.  (The reader may find the proof of Lemma~\ref{L:bound Fmj} helpful in making these modifications.)  

The proof is based on the method of refinements.  We begin by observing that there exists a point $(s_0,x_0) \in E$ and a set $\Omega_d \subset \R^d$ such that $\Phi^d_{(s_0,x_0)}(\Omega_d) \subset F'$ and 
\begin{align} \label{E:83}
\int_{\Omega_d} (2^l)^{\frac{2K(d-1)}{d(d+1)}}(2^{l'})^{\frac{2K}{d(d+1)}}\, dt_{D+1}\, ds_D\, dt_D\, \cdots\, ds_1\, dt_1 \gtrsim \alpha^D\beta^D\beta'.
\end{align}
Furthermore, if $(t_1,s_1,\dots,t_D,s_D,t_{D+1}) \in \Omega_d$, then $t_1,\dots,t_D \in (2^{l-1},2^l]$, $t_{D+1} \in (2^{l'-1},2^{l'}]$, and
\begin{align}
\label{E:84}
|s_i-s_{i-1}| \gtrsim (2^{-l})^{\frac{2K}{d(d+1)}}\alpha, \quad 1 \leq i \leq D\\
\label{E:85}
|t_i-t_j| \gtrsim (2^{-l})^{\frac{2K}{d(d+1)}}\beta, \quad 1 \leq i < j \leq D\\
\label{E:86}
|t_{D+1}-t_i| \gtrsim (2^{-l'})^{\frac{2K}{d(d+1)}}\beta', \quad 1 \leq i \leq D.
\end{align}
The proof of this observation is similar to that of Lemma~\ref{L:mixed odd} and also to one step in the proof of Lemma~2 in \cite{Laghi}, so we omit the details.

We consider first the case then $|l-l'| > 1$.  In this case we have in addition the lower bound 
\begin{equation} \label{E:87}
|t_{D+1}-t_1| \gtrsim 2^{\max\{l,l'\}},
\end{equation}
by purely geometric considerations.  Using Bezout's theorem and Proposition~\ref{P:jacobian bounds}, we have
\begin{align*}
|F'| &\geq |\Phi^d_{(s_0,x_0)}(\Omega_d)| \gtrsim \int_{\Omega_d}|\det D\Phi^d_{(s_0,x_0)}(\Omega_d)|\, dt_{D+1}\, ds_D\, \cdots\, ds_1\, dt_1 \\
&\gtrsim \int_{\Omega_d}\prod_{i=1}^D \bigl\{|s_i-s_{i-1}||t_i|^{\frac{2K}d}\prod_{\stackrel{1 \leq j \leq D+1}{j \neq i}}|t_j-t_i|^2\bigr\}|t_{D+1}|^{\frac{K}d} \, dt_{D+1}\, ds_D\, \cdots\, ds_1\, dt_1.
\end{align*}
Next, using the lower bounds \eqref{E:84}, \eqref{E:85}, and \eqref{E:87}, followed by \eqref{E:83}, we have
\begin{align*}
&|F'| \gtrsim \int_{\Omega_d} (2^l)^{-\frac{2KD}{d(d+1)}-\frac{4KD(D-1)}{d(d+1)} + \frac{2KD}d}(2^{l'})^{\frac{K}d}2^{2D\max\{l,l'\}}\alpha^D\beta^{2D(D-1)}\, dt_{D+1}\, \cdots\, dt_1 \\
&\gtrsim (2^l)^{-\frac{2KD}{d(d+1)}+\frac{2KD}d - \frac{4KD(D-1)}{d(d+1)} - \frac{2K(d-1)}{d(d+1)}}(2^{l'})^{\frac{K}d - \frac{2K}{d(d+1)}}2^{2D\max\{l,l'\}}\alpha^{2D}\beta^{2D(D-1)+D}\beta'.
\end{align*}
Thus, after some arithmetic, we see that 
\begin{equation} \label{E:88}
|F'| \gtrsim (2^l)^{\frac{K(d-1)}{d(d+1)}}(2^{l'})^{\frac{K(d-1)}{d(d+1)}}2^{(d-1)\max\{l,l'\}}\alpha^{d-1}\beta^{\frac{(d-1)(d-2)}2}\beta'.
\end{equation}
Since $\mu_{\theta_0}((2^{k-1},2^k]) \sim 2^{k(1+\frac{2K}{d(d+1)})}$ for $k \in \Z$, we have by \eqref{E:bigger than beta} that 
$$
\beta \lesssim 2^{l(1+\frac{2K}{d(d+1)})} \qquad \beta' \lesssim 2^{l'(1+\frac{2K}{d(d+1)})}.
$$
Thus 
$$
2^{\max\{l,l'\}} \gtrsim 2^{\frac12|l-l'|}(\sqrt{\beta\beta'})^{\frac1{1+2K/d(d+1)}},
$$
as can be seen by (for instance) treating the cases $l' > l$ and $l > l'$ separately.  Plugging these two pieces of information into \eqref{E:88}, we obtain
\begin{align*}
|F'| &\gtrsim \beta^{\frac{K(d-1)}{d(d+1)}\cdot\frac1{1+2K/d(d+1)}+\frac{(d-1)(d-2)}2}(\beta')^{(\frac{K(d-1)}{d(d+1)}+\frac{d-1}2)(\frac1{1+2K/d(d+1)})+1}2^{\frac{d-1}2\max\{l,l'\}}\alpha^{d-1}\\
&\gtrsim 2^{\frac{d-1}4|l-l'|}\beta^{\frac{(d-1)^2}2-\lambda}(\beta')^{\frac{d+1}2+\lambda}\alpha^{d-1},
\end{align*}
which is what we were trying to prove.

Now we turn to the case when $|l-l'| \leq 1$, which is a bit simpler since $2^l \sim 2^{l'}$.  In this case, we use (\ref{E:84}-\ref{E:86}), then \eqref{E:83}, and some algebra (the $2^l$ factors all cancel out) to obtain
\begin{align*}
|F'| &\gtrsim \int_{\Omega_d} (2^l)^{-\frac{2KD}{d(d+1)} - \frac{4KD^2}{d(d+1)} + K} \alpha^D\beta^{2D(D-1)}(\beta')^D\, dt_{D+1}\, \cdots\, ds_1\, dt_1\\
&\gtrsim \alpha^{d-1}(\beta')^d \beta^{\frac{d^2-d+2}2 - d},
\end{align*}
completing the proof of the lemma.
\end{proof}

Finally we move to the proof of Lemma~\ref{L:disjoint}, the last step remaining in the proof of Theorem~\ref{T:main}.

\begin{proof}[Proof of Lemma~\ref{L:disjoint}]
It suffices to prove that 
$$
\sum_{(k,l) \in \scriptK_{\eps\eta}} \sum_{j \in \scriptJ_{kl}} |E_{jkl}| \lesssim (\log(1+\eps^{-1}))^3 |E|.
$$
Furthermore, since $\#\scriptJ_{kl} \lesssim \log(1+\eps^{-1})$ for each $(k,l) \in \scriptK_{\eps\eta}$, it suffices to prove that
\begin{equation} \label{E:sep disj}
\sum_{(k,l) \in \scriptK_{\eps\eta}} \sum_{j \in \scriptJ_{kl}} |E_{jkl}| \lesssim |E|
\end{equation}
under the additional hypotheses that
\begin{equation} \label{E:extra hyp}
\text{$\scriptK_{\eps\eta}$ is $C'\log(1+\eps^{-1})$-separated and $\#\scriptJ_{kl}=1$ for all $(k,l) \in \scriptK_{\eps\eta}$,}
\end{equation}
where $C'$ is some large constant to be determined later on.

The argument we use originated in \cite{ChQex} and was used in \cite{Laghi} in a related context.

We begin by noting that by Cauchy--Schwartz and some elementary manipulations, we have
\begin{align*}
\bigl(\sum_{(k,l)} |E_{jkl}|\bigr)^2 &= \bigl(\int_E \sum_{(k,l)} \chi_{E_{jkl}}\bigr)^2 \leq |E|\int_E \sum_{(k,l),(k',l')}\chi_{E_{jkl}}\chi_{E_{j'k'l'}} \\&= |E| \bigl(\sum_{(k,l)} |E_{jkl}| + \sum_{(k',l') \neq (k,l)} |E_{jkl}\cap E_{j'k'l'}|\bigr),
\end{align*}
where $(k,l) \in \scriptK_{\eps\eta}$ and $\{j\} = \scriptJ_{kl}$ is understood whenever the subscript $jkl$ appears (likewise for $(j',k',l')$).  Thus failure of \eqref{E:sep disj} implies that 
\begin{equation}\label{E:if not disj}
\bigl(\sum_{(k,l)} |E_{jkl}|\bigr)^2 \lesssim |E| \sum_{(k,l) \neq (k',l')}|E_{jkl}\cap E_{j'k'l'}| \leq |E| (\#\scriptK_{\eps\eta})^2 \sup_{(k,l) \neq (k',l')} |E_{jkl}\cap E_{j'k'l'}|.
\end{equation}

Now let $(k,l) \in \scriptK_{\eps\eta}$.  Since $\Pi(F_{jkl}) \subset (2^{l-1},2^l]$, the weight $t^{\frac{2K\theta}{(d+2)(d-1)}}$ is essentially constant, so  \eqref{E:Ejkl sim E} is equivalent to 
\begin{equation} \label{E:Ejkl sim E 2}
\langle X_{\theta_0} \chi_E,\chi_{F_{jkl}}\rangle \sim \langle X_{\theta_0} \chi_{E_{jkl}},F_{jkl}\rangle.
\end{equation}
So by \eqref{E:Fjkl qex}, \eqref{E:Ejkl sim E 2}, and Proposition~\ref{P:rwt}, we have
$$
\begin{aligned}
&\eps^C |E|^{\frac1{p_{\theta_0}}}\|\chi_{F_{jkl}}\|_{L^{q_{\theta_0}'}(L^{r_{\theta_0}'})} \lesssim \langle X_{\theta_0}\chi_E,\chi_{F_{jkl}}\rangle\\
&\qquad \sim \langle X_{\theta_0}\chi_{E_{jkl}},\chi_{F_{jkl}}\rangle \lesssim |E_{jkl}|^{\frac1{p_{\theta_0}}}\|\chi_{F_{jkl}}\|_{L^{q_{\theta_0}'}(L^{r_{\theta_0}'})},
\end{aligned}
$$
and so $\eps^{p_{\theta_0}C}\#\scriptK_{\eps\eta}|E| \lesssim \sum_k |E_{jkl}|$.  Plugging this into \eqref{E:if not disj}, we see that if \eqref{E:sep disj} fails, then there exist distinct ordered pairs $(k,l),(k',l') \in \scriptK_{\eps\eta}$ such that 
\begin{equation} \label{E:bad thing}
\eps^{2p_{\theta_0}C} |E| \lesssim |E_{jkl}\cap E_{j'k'l'}|.
\end{equation}

We assume \eqref{E:bad thing} and will derive a contradiction by applying Lemma~\ref{L:improved Laghi} with $H = E_{jkl} \cap E_{j'k'l'}$, $F = F_{jkl}$, $F' = F_{j'k'l'}$.  By \eqref{E:Fjkl qex}, the definition of $E_{jkl}$, and the fact that $\Pi(F_{jkl}) \subset (2^{l-1},2^l]$, we have for $(s,x) \in H$ that 
$$
X_{\theta_0}^*\chi_F(s,x) \gtrsim \frac{\langle X_{\theta_0}\chi_E,\chi_{F_{jkl}}\rangle}{|E|} \gtrsim \eps^C|E|^{-\frac1{p_{\theta_0}'}}|F_{jkl}|^{\frac1{q_{\theta_0}'}} =: \beta
$$
and similarly that
$$
X_{\theta_0}^*\chi_{F'}(s,x) \gtrsim \eps^C|E|^{-\frac1{p_{\theta_0}'}}|F_{j'k'l'}|^{\frac1{q_{\theta_0}'}} =: \beta'.
$$
Using these bounds and \eqref{E:bad thing}, we have
\begin{align*}
\alpha &:= \frac{\langle X_{\theta_0}\chi_H,\chi_F\rangle}{|F|} \gtrsim \frac{\beta|H|}{|F|} \gtrsim \eps^{C(1+p_{\theta_0})}|E|^{\frac1{p_{\theta_0}}}|F_{jkl}|^{-\frac1{q_{\theta_0}}}.\\
\alpha' &:= \frac{\langle X_{\theta_0}\chi_H,\chi_{F'}\rangle}{|F'|} \gtrsim \eps^{C(1+p_{\theta_0})}|E|^{\frac1{p_{\theta_0}}}|F_{j'k'l'}|^{-\frac1{q_{\theta_0}}}.
\end{align*}
By Lemma~\ref{L:improved Laghi} and a bit of algebra, we have (regardless of the separation between $l$ and $l'$)
\begin{equation} \label{E:Fjkl big}
|F_{jkl}| \gtrsim \eps^{AC} 2^{a|l-l'|}|F_{j'k'l'}|,
\end{equation}
for some constants $A,a > 0$.  By symmetry, we have in addition that
\begin{equation} \label{E:Fj'k'l' big}
|F_{j'k'l'}| \gtrsim \eps^{AC} 2^{a|l-l'|}|F_{jkl}|.
\end{equation}
Obviously, this implies that $|l-l'| \leq C\log(1+\eps^{-1})$.  Thus if $C'$ is sufficiently large, our hypothesis \eqref{E:extra hyp} implies that $|k-k'| > \frac{C'}2 \log(1+\eps^{-1})$.  But by \eqref{E:Fjkl same size}, we have
\begin{gather*}
\eps^C2^{l\frac{2K}{d(d+1)}}(2^{-k}\eta)^{q_\theta'}|E|^{\frac{2(d-1)}{d(d+1)}} \lesssim |F_{jkl}|^{\frac1{q_{\theta_0}'}} \lesssim \eps^{-C}2^{l\frac{2K}{d(d+1)}}(2^{-k}\eta)^{q_\theta'}|E|^{\frac{2(d-1)}{d(d+1)}}\\
\eps^C2^{l\frac{2K}{d(d+1)}}(2^{-k'}\eta)^{q_\theta'}|E|^{\frac{2(d-1)}{d(d+1)}} \lesssim |F_{j'k'l'}|^{\frac1{q_{\theta_0}'}} \lesssim \eps^{-C}2^{l\frac{2K}{d(d+1)}}(2^{-k'}\eta)^{q_\theta'}|E|^{\frac{2(d-1)}{d(d+1)}},
\end{gather*}
where we have used the bound $|l-l'| \lesssim \log(1+\eps^{-1})$ to eliminate $l'$.  These bounds are incompatible with \eqref{E:Fjkl big}, \eqref{E:Fj'k'l' big} and the fact that $|k-k'| > \frac{C'}2 \log(1+\eps^{-1})$ (for $C'$ sufficiently large).  This completes the proof of the lemma.
\end{proof}

%%%%%%%%%%%%%%%%%%%%%%%%%%%%%%%%%%%%%%%%%%%%%%%%%%%%%%%%%%%%%%%%%%%%%%%%%%%%%%%%%%%%%%%%%%%%%%%%%%%%%%%%%%%%%%%%%%%%%%%%%%%%%%%%%%%%%%%%%%%%%%%%%%%%%%%%%%%%%%%%%%%%%%

\end{document}